\RequirePackage{xcolor}
\documentclass[journal,twoside,web]{ieeecolor}
\usepackage{generic}
\usepackage{cite}
\usepackage{amsmath,amssymb,amsfonts}
\usepackage{algorithmic}
\usepackage{graphicx}
\usepackage[hidelinks]{hyperref}
\usepackage{textcomp}
\usepackage{systeme}
\usepackage{amsmath}
\usepackage{amssymb}
\usepackage{dsfont}
\usepackage{algorithm}
\usepackage{booktabs}
\usepackage[caption=false,font=footnotesize]{subfig}
\usepackage{tikz}
\usetikzlibrary{arrows.meta,calc,positioning}

\newcommand{\GenericHyperbolicVertex}[1]{%
\begin{scope}[shift={#1}]
  \draw[red,very thick,->] (0,0.7) -- (4,0.7);
  \draw[blue,very thick,<-,] (0,-0.7) -- (4,-0.7);

  \draw[blue,>=stealth,thick]
  (-0.45,0) arc (-180:-135:1.0);
  \draw[blue,>=stealth,->,thick]
  (-0.45,0) arc (-180:-225:1.0);
  \node[blue,left] at (-0.45,0) {$K_{(i,+1),(i,+1)}$};

  \draw[red,>=stealth,thick]
  (4.45,0) arc (0:45:1.0);
  \draw[red,>=stealth,->,thick]
  (4.45,0) arc (0:-45:1.0);
  \node[red,right] at (4.45,0) {$K_{(i,-1),(i,-1)}$};

  \draw[dashed,<-] (1.4,-0.7) -- (1.4,0.7);
  \draw[dashed,->] (2.6,-0.7) -- (2.6,0.7);
  \node[left] at (1.4,0) {$\Sigma_i^-$};
  \node[right] at (2.6,0) {$\Sigma_i^+$};

  \node[red,above] at (2,0.7) {$w_i^+(t,x)$};
  \node[blue,below] at (2,-0.7) {$w_i^-(t,x)$};

  \node[above left]  at (0,0.7) {$r_i^{+1}$};
  \node[below left]  at (0,-0.7) {$y_i^{+1}$};

  \node[above right] at (4,0.7) {$y_i^{-1}$};
  \node[below right] at (4,-0.7) {$r_i^{-1}$};
\end{scope}
}

\newcommand{\PortVertex}[4]{%
\begin{scope}[shift={({#1},{#2})},scale=#4]
  \draw[thick,fill=white] (0,0) circle (0.45);
  \node at (0,0) {$#3$};

  \filldraw[black,fill=white,thick] (-0.55,0) circle (0.10);
  \filldraw[black,fill=white,thick] ( 0.55,0) circle (0.10);

  \node[above] at (-0.80,0.10) {$p_{#3}^{+1}$};
  \node[above] at ( 0.90,0.10) {$p_{#3}^{-1}$};

\end{scope}
}

\usepackage{mathtools}

\newtheorem{thm}{Theorem}[section] \newtheorem{assum}{Assumption}[section] \newtheorem{prop}{Proposition}[section] \newtheorem{defn}{Definition}[section]  \newtheorem{rem}{Remark}[section]   \newtheorem{corollary}{Corollary}[section] \newtheorem{lem}{Lemma}[section]

\newcommand{\ind}{\mathds{1}}
\DeclareMathAlphabet{\mathbbb}{U}{bbold}{m}{n}
\newcommand{\qhat}{\widehat q}
\newcommand{\phat}{\widehat p}
\newcommand{\rank}{\operatorname{rank}}
\renewcommand{\leq}{\leqslant}
\renewcommand{\le}{\leqslant}
\renewcommand{\geq}{\geqslant}

\newcommand{\mycomment}[1]{}
\newcommand{\Ntarget}{N_0}

\newcommand{\Nhat}{\widehat{N}}

\newcommand{\Mhat}{\widehat{M}}
\newcommand{\col}{\operatorname{col}}
\newcommand{\diag}{\operatorname{diag}}
\newcommand{\Bgraph}{B_g}

\newcommand{\R}{\mathbb{R}}
\newcommand{\N}{\mathbb{N}}

\newcommand{\C}{\mathbb{C}}

\newcommand{\Lap}[1]{\widehat{#1}}
\newcommand{\dmax}{\Bar{\tau}}
\newcommand{\infpp}{\delta_0}
\newcommand{\dynavar}{Y}
\newcommand{\dynaopvar}{\mathfrak A}
\newcommand{\dynaopstatevar}{\mathfrak B}
\newcommand{\feedbackdynaop}{\mathcal P}
\newcommand{\spaceY}{\mathcal Y}
\newcommand{\normyinit}{\mathcal Y}
\newcommand{\gains}{\mathcal G}
\newcommand{\graph}{\Gamma}
\newcommand{\gainspace}{\mathfrak{G}}

\newcommand{\statespace}{\mathcal{X}}

\renewcommand{\star}{*}

\DeclareMathOperator{\myRe}{Re}
\DeclareMathOperator{\myIm}{Im}
\renewcommand{\Re}{\myRe}
\renewcommand{\Im}{\myIm}
\def\BibTeX{{\rm B\kern-.05em{\sc i\kern-.025em b}\kern-.08em
    T\kern-.1667em\lower.7ex\hbox{E}\kern-.125emX}}

\begin{document}
\title{Stabilization of 1D Linear Hyperbolic Balance Laws by Integral Difference Control and Application to Networks Stabilization}
\author{Adam Braun, Jean Auriol, Lucas Brivadis
\thanks{This project has received funding from the Agence Nationale de la Recherche (ANR) via grant PANOPLY ANR-23-CE48-0001-01.}
\thanks{All authors are with Universit\'e Paris-Saclay, CNRS, CentraleSup\'elec, Laboratoire des signaux et syst\`emes, 91190, Gif-sur-Yvette, France. {\tt\small firstname.lastname@centralesupelec.fr}
}
}
\maketitle
\thispagestyle{empty}
\begin{abstract}
This paper develops a unified method for the exponential stabilization of first-order linear hyperbolic balance laws under general actuation, including both underactuated boundary and in-domain control. The proposed framework brings together a wide range of underactuated configurations within a single formulation and substantially extends existing results restricted to particular actuation settings.
Using cutting and folding transformations, the proposed approach is further applied to networks of hyperbolic balance laws, including configurations with cycles. The control design is developed under a stabilizability condition and a robustness assumption. It is based on an invertible backstepping transformation, which partially decouples the system, followed by a reformulation of the stabilization problem at the level of an Integral Difference Equation (IDE). The gains of the resulting dynamic feedback law are constructed at the IDE level by combining a stable rank-reduction procedure, which reduces the problem to a single-input design, with the solution of an interpolation equation arising from a Corona problem. Numerical simulations are presented for a relevant cycle network that cannot be addressed by existing methods in the literature.
\end{abstract}

\section{Introduction}
\label{sec:introduction}
Hyperbolic Partial Differential Equations (PDEs) constitute a fundamental class of models for systems in which information, mass, momentum, or energy propagates at finite speed. Owing to this intrinsic transport structure, they arise naturally in the description of many large-scale engineering systems, such as open-channel hydraulic networks~\cite[Chapter 8]{bastin_coron}, traffic flow models~\cite{YuBook2023}, gas transportation networks, and drilling processes~\cite{Aamo2016,auriol2020closed}. In many of these applications, actuation is exerted through the boundaries of the spatial domain. This is the case, for instance, in ramp-metering strategies for traffic regulation~\cite{YuBook2023,zhang2024mean} and in the suppression of torsional vibrations in drilling systems through boundary actuation~\cite{Aarsnes2018b,auriol2022comparing}. 
{A large part of the literature considers configurations where full actuation is available at one boundary, as in~\cite{bastin_coron}.}
{However, several relevant configurations~\cite{Jean_Adam_Ouidir_Mathieu} involve underactuated boundary control and control actions distributed inside the domain.} Examples include tubular chemical reactors, and in particular plug-flow reactors, where the jacket temperature acts as a spatially distributed manipulated variable~\cite{orlov2002discontinuous,pitarch2016distributed}. Distributed actuation also appears in fluid-filled tanks modeled by Saint-Venant equations, where the force applied to the tank enters the dynamics as a distributed scalar input~\cite{dubois1999motion}. More recently, similar control structures have been identified in mixed traffic flow models involving both manually driven and adaptive cruise control equipped vehicles, where the time-gap parameter of connected vehicles can be interpreted as a distributed input along the road segment~\cite{bekiaris2020pde}.

Among constructive methods, backstepping has become one of the most prominent tools for the stabilization of hyperbolic PDEs~\cite{krstic2008boundary}. The central idea is to transform the original system, through an appropriate Volterra integral change of variables, into a target system with prescribed stability properties~\cite{krstic2008boundary,redaud2024domain}. This framework has led to exponential and finite-time stabilization results for many classes of coupled hyperbolic systems~\cite{auriol2016minimum,coron2017finite,coron2013local}. It has also been extended to robust control~\cite{auriol2020robust}, finite-time regulation~\cite{deutscher17}, and adaptive control~\cite{Anfinsen_book2019}. More recent contributions have considered interconnected structures, including PDE-ODE cascades and networks of coupled PDEs~\cite{auriol2024output,deutscher2019output,di2018stabilization,irscheid2021observer,redaud2024output}. Recent developments in backstepping have also introduced approaches based on Fredholm transformations~\cite{deutscher2019fredholm,CoronLu2015Fredholm,redaud2022stabilizing,gagnonbackstepfred,GagnonHayatMarxXiangZhang2026}. These transformations offer greater flexibility than classical Volterra transformations, but their invertibility is not automatic. To overcome this difficulty, much of the existing literature assumes that the differential operator arising in the abstract formulation of the PDE is diagonalizable in a Riesz basis, which provides a sufficient framework for proving the invertibility of the Fredholm transformation. 
A detailed account of recent developments in PDE backstepping can be found in the survey~\cite{vazquez2024backstepping}.

{We establish a unified exponential stabilization methodology for first-order linear hyperbolic balance laws that covers the full range of partial actuation configurations, whether located at the boundary, within the spatial domain, or both.}
This flexibility makes it possible to treat a variety of underactuated configurations and thereby extends several existing results devoted to specific cases of interest. These include, among others, the classical setting of full actuation at one boundary, extensively studied through backstepping methods~\cite{auriol2016minimum,coron2017finite}; underactuated boundary control problems~\cite{Jean_Adam_Ouidir_Mathieu}; two-sided actuation frameworks as considered in~\cite{auriol2018two}; and in-domain control problems, recently investigated in the scalar case in~\cite{AuriolSIAM}.

Furthermore, by relying on cutting and folding transformations in the spirit of~\cite{folding_vazquez, auriol2022folding}, we show that broad classes of networks (including \emph{cycles}) of first-order hyperbolic balance laws can be transformed into the general system considered in this work. This provides a direct extension of previous contributions on the stabilization of hyperbolic PDE networks~\cite{auriol2024output, auriol_hdr, braun_stab_3_2inputs, braun2025stabilizationchainhyperbolicpdes}.
This is of particular interest given the relevance of PDE networks in the modeling of many real-world systems~\cite{HAYAT2026103942, ammari2025graphgeometriccontrolcondition, Nicaise, Jacek}. Most existing works on this topic are limited to tree-like structures (see e.g.~\cite{AMMARI201939,baudouin2015inverse}) and assume that no cycles are present in the underlying graph. The method developed in this paper partially overcomes this limitation by imposing instead a stabilizability condition on the network.

Under a stabilizability condition and a necessary robustness assumption, we establish exponential stability of the closed-loop system by constructing a dynamic state-feedback control law. The methodology proceeds as follows. First, inspired by~\cite{auriol2019explicit}, the original system is mapped, through an invertible backstepping transformation, into a partially decoupled target system. Then, inspired by~\cite{auriol_hdr}, we establish that exponential stability of this target PDE system is equivalent to exponential stability of an associated Integral Difference Equation (IDE). Building on recent developments in~\cite{braun_corona}, the stabilizing gains are constructed at the IDE level. More precisely, a stable rank-reduction procedure is applied to reduce the problem to a single-input control problem while preserving stabilizability properties; see also~\cite{quadrat2004general} for related algebraic rank conditions over rings of delay operators. The gains are then obtained \{following the methodology introduced in~\cite{braun_corona} in the case of scalar input, by solving an interpolation equation arising from a Corona problem~\cite{Carleson1962}.

The article is organized as follows. Section~\ref{section:problem_formulation} introduces the class of systems considered in this work, discusses how this formulation generalizes previous contributions, and states the stabilization objective and the main assumptions. Section~\ref{section:control_design} presents the backstepping target system and the proposed dynamic state-feedback control law, before stating the main result. Section~\ref{section:IDE} reformulates the stabilization problem as an IDE, while Section~\ref{section:design_command_law} describes the construction of the stabilizing gains at the IDE level. Section~\ref{section:network} applies the proposed framework to networks of linear hyperbolic balance laws by means of cutting and folding transformations, relying in particular on an associated signed graph~\cite{harary1953notion_graph}. Section~\ref{section:numerics} is devoted to numerical simulations of an underactuated cycle network controlled through one boundary input and one in-domain input.  The technical proofs are collected in the appendices. Appendix~\ref{appendix:notations} summarizes the notation used throughout the paper.
\smallskip

\noindent
\textbf{Notation.}
$\mathbb{R}_{>0}^n$ denotes $(0,+\infty)^n$.
$\operatorname{Id}_d$ denotes the $d\times d$ identity matrix.
$L^2$ denotes the space of square-integrable functions.
$H^1$ denotes the first-order Sobolev space.
$C((0,\infty),B)$ denotes the space of continuous functions from $(0,\infty)$ to a Banach space $B$.
When properly defined, $\Lap N$ denotes the Laplace transform of $N$.
$\mathbb{C}_\omega$ denotes the half-plane $\{s\in\mathbb{C}:\operatorname{Re}(s)>-\omega\}$, where $\omega>0$. The adjoint of an operator or a matrix $A$ is denoted by $A^\star$. For any $\omega>0$ the weighted space $L^2_\omega((0,\infty),\R)$ is the space of real-valued measurable functions such that
$
\|f\|_{L^2_\omega} := \left(\int_0^\infty e^{2\omega t}|f(t)|^2 dt\right)^{1/2}<\infty.$ {Let $V$ be a vector space. We denote by $V'$ its algebraic dual, and by $\left\langle \cdot, \cdot\right\rangle_{V', V}$ the associated duality bracket.}
All important variable notations are summarized in Appendix~\ref{appendix:notations}.
\section{Problem Formulation}
\label{section:problem_formulation}
\subsection{System considered}
We consider general systems of linear hyperbolic balance laws defined for  all $t>0$ and all $x\in [0,1] $ by
\begin{equation} \label{eq:hyperbolic_couple}
\left\{
\begin{aligned}
&\partial_tw(t, x)+\Lambda \partial_xw(t, x)=\Sigma(x) w(t,x) + h(x)U(t), \\
&w^+(t, 0)=Q w^-(t, 0) + B_0U(t),\\
&w^-(t, 1)=R w^+(t, 1)+B_1U(t),
\end{aligned}
\right.
\end{equation}
where $w=\col({w^+}, {w^-})\in\mathbb{R}^{n+m}$ denotes the state, with $n, m \in \N^*$. The components $w^+ \in\mathbb{R}^{n}$ and $w^- \in\mathbb{R}^{m}$ represent the rightward and leftward propagating states, respectively. The diagonal velocity matrix is
\[
\Lambda=\operatorname{diag}(\Lambda^+, -\Lambda^-), 
\]
with 
\[
\Lambda^+ = \operatorname{diag}(\lambda_1,\dots,\lambda_{n}), \quad \Lambda^- = \operatorname{diag}(\mu_1,\dots,\mu_m),
\]
and ordered entries $-\mu_m<\cdots<-\mu_1<0<\lambda_1<\cdots<\lambda_{n}$. 
The case of isotachic blocks (i.e., multiple states sharing the same transport speed) is not treated here, but the techniques of~\cite{hu2016control,vazquez2024backstepping} can be used to rewrite such systems in the present framework.
The continuous matrix-valued function $\Sigma$ captures in-domain couplings and is partitioned as
\begin{equation}
    \Sigma(x)=\begin{pmatrix}
        \Sigma^{++}(x)&  \Sigma^{+-}(x)\\[1mm]
        \Sigma^{-+}(x) &  \Sigma^{--}(x)
    \end{pmatrix},
\end{equation}
with $\Sigma^{++}(x) \in \mathbb{R}^{n\times n}$, $\Sigma^{+-} (x)\in \mathbb{R}^{n\times m}$, $\Sigma^{-+} (x)\in \mathbb{R}^{m\times n}$, and $\Sigma^{--} (x)\in \mathbb{R}^{m\times m}$.
Following~\cite{hu2016control}, we assume without loss of generality that $\Sigma_{ii}=0$ for all $1\leq i\leq n+m$.
The matrices $Q\in\mathbb{R}^{n\times m}$ and $R\in\mathbb{R}^{m\times n}$ correspond to constant boundary couplings. The control input $U(t)\in \mathbb{R}^{d}$ acts on the plant~\eqref{eq:hyperbolic_couple} through the boundary via $B_0 \in \mathbb{R}^{n\times d}$ and $B_1 \in \mathbb{R}^{m\times d}$, and through the domain via the continuous matrix-valued function $h:[0,1]\to\mathbb{R}^{(n+m)\times d}$, partitioned as
\begin{equation}
    h(x)=\begin{pmatrix}
        h^+(x)\\[1mm]
        h^-(x) 
    \end{pmatrix},
\end{equation}
with $h^+(x) \in \R^{n \times d}$ and $h^-(x) \in \R^{m \times d}$. 

The formulation~\eqref{eq:hyperbolic_couple} is deliberately general, extending the classical boundary-control setting for heterodirectional hyperbolic systems in several directions. It encompasses, as special cases, multiple configurations studied in the literature:
\begin{itemize}
  \item \emph{Full actuation at one boundary.}
    Setting $h=0$, $B_0=0$, $d=m$, and $B_1=\operatorname{Id}_m$ yields a boundary stabilization problem with full actuation of all incoming characteristics at the right boundary, broadly studied  via backstepping~\cite{auriol2016minimum,coron2017finite}.

  \item \emph{Underactuated boundary control.} Setting $h=0$ and $B_0=0$ with $d<m$ corresponds to an underactuated boundary configuration~\cite{Jean_Adam_Ouidir_Mathieu}.

  \item \emph{Two-sided actuation.}
    Setting $h=0$, $d=n+m$, $B_0 = \bigl[\operatorname{Id}_n \;\big|\; 0_{n\times m}\bigr]$,
    and $B_1 = \bigl[0_{m\times n} \;\big|\; \operatorname{Id}_m\bigr]$
    leads to the two-sided actuation framework treated in~\cite{auriol2018two}.

  \item \emph{In-domain actuation.} Setting $B_0=0$ and $B_1=0$ yields a purely in-domain control problem, recently studied in the scalar case in~\cite{AuriolSIAM}.
\end{itemize}
Beyond these special cases, the formulation~\eqref{eq:hyperbolic_couple} accommodates underactuated mixed configurations, in which a single finite-dimensional input acts simultaneously through the left boundary, the right boundary, and the distributed domain, possibly in a partial or redundant manner. Such configurations lie beyond the reach of classical backstepping designs, which typically require one-sided or fully actuated boundary conditions. This level of generality is motivated by the network application developed in Section~\ref{section:network}, where cutting and folding transformations naturally yield boundary and distributed input structures that fall outside the classical fully actuated framework. The present approach therefore applies to a broad class of hyperbolic PDE networks and extends the scope of existing works~\cite{auriol_hdr,braun_stab_3_2inputs,braun2025stabilizationchainhyperbolicpdes}.

Defining the state space $\statespace:=L^2((0,1),\R^{n+m})$,
system~\eqref{eq:hyperbolic_couple} can be formulated in the following abstract form
\begin{equation}
\label{eq:edp_coupled_operator_formuled}
\frac{d}{dt}w(t) = \mathcal A w(t) + \mathcal B U(t),
\end{equation}
where the operator~$\mathcal{A}:D(\mathcal{A}) \subset \statespace \to \statespace$ is defined by
\[
\mathcal A :w \longmapsto -\Lambda w_x+\Sigma(x)w,
\]
with domain
\begin{align*}
D(\mathcal A):=\{
&w=(w^+, w^-)\in H^1\big((0,1);\mathbb R^{n+m}\big)\;:\;\\
&
w^+(0)=Qw^-(0),~ w^-(1)=Rw^+(1)
\}.
\end{align*}
Its adjoint $\mathcal{A}^\star:D(\mathcal{A}^\star) \subset \statespace \to \statespace $ is given by 
\[
\mathcal A^\star: z \longmapsto \Lambda z_x+\Sigma(x)^\star z,\quad \forall z\in D(\mathcal A^\star)
\]
with domain
\begin{align*}
&D(\mathcal A^\star):=
\{
z=(z^+, z^-)\in H^1\big((0,1);\mathbb R^{n+m}\big)\;:\;\\
&\Lambda^- z^-(0)=Q^\star \Lambda^+ z^+(0),~
\Lambda^+ z^+(1)=R^\star \Lambda^- z^-(1)
\}.
\end{align*}
 {The operator $\mathcal{B}: \mathbb{R}^p \to D(\mathcal A^\star)'$ is defined through its adjoint}  $\mathcal{B}^\star: D(\mathcal{A}^\star) \to  \mathbb{R}^p$:
\begin{align*}
&\mathcal B^\star: z \mapsto
B_0^\star\Lambda^+ z^+(0)
+
B_1^\star\Lambda^- z^-(1)
\\
&+
\int_0^1 \left(h^+(x)^\star z^+(x)+h^-(x)^\star z^-(x)\right)\,dx, \quad \forall z\in D(\mathcal A^\star).
\end{align*}

{The well-posedness of the closed-loop system associated to~\eqref{eq:hyperbolic_couple} will be shown once the control law $U(t)$ is properly introduced.}

\subsection{Stabilization Objective}
{The objective is to design a control law that exponentially stabilizes system~\eqref{eq:hyperbolic_couple}. Since attempts based on static feedback did not prove successful, we instead propose a dynamic state-feedback control law.
More precisely, we aim to find  a Banach space $\spaceY$ with norm $\|\cdot\|_{\normyinit}$, a variable $\dynavar$, operators $\dynaopstatevar, \dynaopvar, \feedbackdynaop$  {defined on dense subsets of $\spaceY$} or $\statespace$ with
\begin{equation}
    \label{eq:feedback_abstract_form}
    \dot \dynavar(t) = \dynaopvar \dynavar(t)+ \dynaopstatevar w(t),  \quad U(t) = \feedbackdynaop\dynavar(t),
\end{equation}
and such that the closed-loop~\eqref{eq:hyperbolic_couple} and \eqref{eq:feedback_abstract_form} is exponentially stable in the following sense.
\begin{defn}[$L^2$-exponential stability] \label{Def_stability}
Provided it exists, the zero solution of~\eqref{eq:hyperbolic_couple} and \eqref{eq:feedback_abstract_form} is \emph{exponentially stable} if there exist $\kappa_0\geq 1,\omega>0$ such that, for every initial condition $(w_0,Y_0)\in \statespace\times \spaceY $, for all $t\geq 0$,
\[\|w(t,\cdot)\|_{L^2} + \|Y(t,\cdot)\|_{\spaceY} \leq \kappa_0\,\mathrm{e}^{-\omega t}\big(\|w_0\|_{L^2} + \|Y_0\|_{\spaceY} \big).
\]
\end{defn}
}

\subsection{Assumptions}
To achieve our stabilization objective, we will rely on the following assumptions.
{The first assumption imposes an open-loop stability condition needed for delay-robust exponential stabilization, as established in~\cite[Theorem~8]{auriol2019explicit}.
\begin{assum} \label{Assum_OL_stab}
    The open-loop system~\eqref{eq:hyperbolic_couple} with $\Sigma=0$ and $U\equiv 0$ is exponentially
    stable in the sense that there exist $\kappa_0\geq1$ and $\omega>0$ such that, for every initial condition $w_0$ in $\statespace$,
\[
\|w(t,\cdot)\|_{L^2} \leq \kappa_0\,\mathrm{e}^{-\omega t}\|w_0\|_{L^2}, \qquad t\geq 0.
\]
\end{assum}
Equivalently, Assumption~\ref{Assum_OL_stab} rules out asymptotic sequences of eigenvalues with non-negative real parts for the operator~$\mathcal A$ {(with $\Sigma =0$)}~\cite{halebook,auriol2019explicit}. This condition is necessary for delay-robust exponential stabilization: if the open-loop transfer function has infinitely many poles in the closed right half-plane, then no positive delay-robustness margin can be achieved in closed loop~\cite{logemann1996conditions}.} 
In addition, we {assume}
the following stabilizability assumption {(Fattorini–Hautus condition, see e.g.~\cite{fattorini1966some})}.
\begin{assum} \label{assum:stabilizability_edp}
    There exists \(\omega>0\) such that
\begin{equation}
\label{eq:pde-stab-ass}
\ker(s-\mathcal A^\star)\cap \ker \mathcal B^\star=\{0\},
~\forall s\in\C_\omega.
\end{equation}

\end{assum}
\section{Control Design and Main Result}
\label{section:control_design}
{In this section, we design a dynamic stabilizing feedback law of the form~\eqref{eq:feedback_abstract_form}. The structure of the control law is obtained through an invertible change of variables, namely a backstepping transformation. We conclude this section by stating the main result of the article.}

\subsection{Backstepping Change of Variables}
The control law will be expressed as a dynamic feedback of $\gamma (t)= \mathcal T w(t)$ where $\mathcal T$ is an invertible change of variable and $\gamma $ is a solution to the following partially decoupled system, defined for all $t>0$, and all $x\in [0,1]$, as
\begin{equation} \label{eq:hyperbolic_target_2}
\left\{
\begin{aligned}
&\gamma_t(t, x)+\Lambda\gamma_x(t, x)=
  \begin{pmatrix}
    G(x)\beta(t,0)\\ H(x)\beta(t, 1)
\end{pmatrix} + h_\gamma(x) U(t),\\[1mm]
&\alpha(t, 0)=Q\beta(t, 0) + B_0U(t),\\
 &\beta(t, 1)=R\alpha(t, 1)+B_1U(t)\\
 &+\int_0^1 F_\alpha(y)\,\alpha(t,y)\,dy +\int_0^1 F_\beta(y)\,\beta(t,y)\,dy,
\end{aligned}
\right.
\end{equation}
where $\gamma(t) :=(\alpha(t), \beta(t)) = \mathcal Tw(t)$, with $\alpha(t,x)\in \R^n, \beta(t,x) \in \R^m$, $G, H, F_\alpha, F_\beta, h_\gamma$ are piecewise continuous functions. 
\begin{lem}
    \label{lem:backsteppin_transform}
    There exist $G, H, F_\alpha, F_\beta, h_\gamma$ piecewise continuous functions and a bounded transformation $\mathcal T: \statespace\to \statespace$ (called the backstepping transform), that is invertible with bounded inverse, and that maps any solution $w$ of~\eqref{eq:hyperbolic_couple} to a solution $\gamma := \mathcal{T}w$ of~\eqref{eq:hyperbolic_target_2}.
    {In particular, there exist $C_1>0, C_2>0$ such that for all $t\geq 0$,
$$C_1\|\gamma(t,\cdot)\|_{L^2}\leq \|w(t,\cdot)\|_{L^2} \leq C_2\|\gamma(t,\cdot)\|_{L^2}.$$}
\end{lem}
The proof of Lemma~\ref{lem:backsteppin_transform} is given in Appendix~\ref{section:backstep_transform}.
As for the original system~\eqref{eq:edp_coupled_operator_formuled}, the target
system~\eqref{eq:hyperbolic_target_2} admits the following abstract formulation
\begin{equation} \label{eq:operator_form_targer_2_gamma}
\frac{d}{dt}\gamma(t) = \mathcal{A}_1\,\gamma(t) + \mathcal{B}_1\,U(t),
\end{equation}
with $\mathcal{A}_1:D(\mathcal{A}_1)\subset\mathcal{X}\to\mathcal{X}$,
\begin{equation} \label{eq:A-def}
\mathcal{A}_1:\binom{\alpha}{\beta}
\mapsto
\binom{-\Lambda^+\alpha_x + G(\cdot)\beta(0)}{\Lambda^-\beta_x + H(\cdot)\beta(1)},
\end{equation}
with domain
\begin{align*}
D(\mathcal{A}_1) = \{
(\alpha,\beta)\in H^1((0,1),\R^{n+m}):
\alpha(0)=Q\beta(0),\\
\beta(1)=R\alpha(1)+\displaystyle\int_0^1\bigl(F_\alpha(\nu)\alpha(\nu)+F_\beta(\nu)\beta(\nu)\bigr)\,d\nu
\}.
\end{align*}
For $(\phi,\psi)\in H^1((0,1);\mathbb{R}^n)\times H^1((0,1);\mathbb{R}^m)$, define 
\begin{equation} \label{eq:zeta-def}
\zeta(\phi, \psi) := \Lambda^-\psi(1) + \int_0^1 H(x)^\star\psi(x)\,dx \in \mathbb{R}^m.
\end{equation}
The adjoint $\mathcal{A}_1^\star:D(\mathcal{A}_1^\star)\subset\mathcal{X}\to\mathcal{X}$
is given by
\begin{equation} \label{eq:Astar-def}
\mathcal{A}_1^\star\binom{\phi}{\psi}
= \binom{\Lambda^+\phi_x + F_\alpha(\cdot)^\star\zeta(\phi, \psi)}{-\Lambda^-\psi_x + F_\beta(\cdot)^\star\zeta(\phi, \psi)},
\end{equation}
with domain
\begin{align*} 
D(\mathcal{A}_1^\star) = \{
&(\phi,\psi)\in H^1((0,1),\R^{n+m}):
\\&Q^\star\Lambda^+\phi(0) - \Lambda^-\psi(0) + \displaystyle\int_0^1 G(x)^\star\phi(x)\,dx = 0,\\
&\Lambda^+\phi(1) = R^\star\zeta(\phi, \psi)\}.
\end{align*}
The adjoint control operator $\mathcal{B}_1^\star:D(\mathcal{A}_1^\star)\to\mathbb{R}^d$ is
\begin{equation} \label{eq:B1star-def}\begin{aligned} 
\mathcal{B}_1^\star:&(\phi,\psi)
\mapsto B_0^\star\Lambda^+\phi(0) + B_1^\star\zeta(\phi, \psi)
\\
&+ \int_0^1\bigl(h_\alpha(x)^\star\phi(x) + h_\beta(x)^\star\psi(x)\bigr)\,dx,
\end{aligned}
\end{equation}
with $\zeta(\phi, \psi)$ given by~\eqref{eq:zeta-def}.
Naturally,
the stabilizability property given in Assumption~\ref{assum:stabilizability_edp} is preserved under the backstepping transformation $\mathcal T$.
\begin{prop} \label{prop:stabilizability_target2}
Let $\omega>0$, Assumption~\ref{assum:stabilizability_edp}  holds if and only if
\begin{equation} \label{eq:pde-stab-ass_target2}
\ker(s-\mathcal{A}_1^\star)\cap\ker\mathcal{B}_1^\star = \{0\},~\forall\,s\in\C_\omega.
\end{equation}
\end{prop}
\begin{proof}
Since the backstepping transformation $\mathcal{T}$ in
Lemma~\ref{lem:backsteppin_transform} is boundedly invertible, the
stabilizability of~\eqref{eq:hyperbolic_couple} under
Assumption~\ref{assum:stabilizability_edp} transfers directly
to~\eqref{eq:operator_form_targer_2_gamma}. Specifically, {because ${\mathcal{T}^{-1}}^\star D(\mathcal A_1^\star) \subset D(\mathcal A^\star)$}, any $(\phi, \psi)\in D(\mathcal A_1^\star)$ such that
$(s-\mathcal{A}_1^\star)(\phi,\psi)=0$ with $\mathcal{B}_1^\star(\phi,\psi)=0$ is
mapped via ${\mathcal{T}^{-1}}^\star$ to a non-trivial kernel element of
$(s-\mathcal{A}^\star)\cap\ker\mathcal{B}^\star$, which is empty
by~\eqref{eq:pde-stab-ass}. Hence $(\phi,\psi)=0$. The other implication is proved in a similar manner.
\end{proof}
\subsection{Control Law Structure}
\label{section:command_law_structure}
The control law is expressed in terms of $\gamma(t) = (\alpha(t), \beta(t) )=\mathcal T w(t)$. Consider the dynamical system defined for all $t>0$ and all $x\in [0,1]$ as
\begin{equation}
\left\{
\begin{aligned}
&\partial_t\chi_i(t,x)+\frac{1}{S_i}\partial_x\chi_i(t,x)
=0,
\quad i=1,\ldots,m,
\\
&\chi_i(t,0)
=\beta_i(t,1), 
\\
&\partial_t\psi(t,x)+\frac{1}{S_{m+1}}\partial_x\psi(t,x)
=0,
\\
&\psi(t,0)
=U_1(t),
\\[1mm]
&\partial_t\phi_j(t,x)+\frac{1}{T_j}\partial_x\phi_j(t,x)
=0,
\quad j=2,\ldots,d,
\\
&\phi_j(t,0)
=\beta(t,1),
\\
&U_1(t)
=
\sum_{i=1}^m
S_i\int_0^1 g_i(S_i x)\chi_i(t,x)\,dx
\\
&+
S_{m+1}\int_0^1 f(S_{m+1}x)\psi(t,x)\,dx,
\\
&U_j(t)
=
\sum_{k=1}^{j-1}v_{jk}U_k(t)
+
T_j u_j^\star\int_0^1\phi_j(t,x)\,dx,~
j=2,\ldots,d.
\end{aligned}
\right.
\label{eq:dynamical_feedback}
\end{equation}

where $\chi, \psi, \phi$ are intermediate PDE states.
The controller gains are
$$\gains :=(S, g, f , u, v, T)\in  \gainspace,$$
with 
\begin{equation} \begin{aligned}\label{eq:gainspace}\gainspace& :=  \R_{>0}^{m+1}\times\prod_{i=1}^m L^2\bigl((0,S_i),\mathbb{R}\bigr)\times L^2\bigl((0,S_{m+1}),\mathbb{R}\bigr)\\
&\times \R^{m\times (d-1)} \times \R^{d \times d}\times \R_{>0}^{d-1}.\end{aligned}\end{equation}
where \(v\in\mathbb R^{d\times d}\) is strictly lower triangular
Using the method of characteristics, the control operator in~\eqref{eq:dynamical_feedback} is equivalent to the following autoregressive ({meaning the control value at time $t$ depends on its own past values, as in~\cite{ammari2024prescribing}) control law, defined for all $j\in \{2,\cdots, d\}$ as
\begin{equation} \label{eq:control_law_auto_regressiv} \left\{
    \begin{aligned}
        U_j(t) &= \sum_{k=1}^{j-1} v_{jk} U_k(t) + u_j^\star \int_0^{T_j} \beta(t-\nu, 1)\,d\nu, \\
        U_1(t) &= \sum_{i=1}^m \int_0^{S_i} g_i(\eta) \beta_i(t-\eta,1)\,d\eta \\
        &~+ \int_0^{S_{m+1}} f(\eta) U_1(t-\eta)\,d\eta.
    \end{aligned}
    \right.
\end{equation}
Notice that the control law~\eqref{eq:dynamical_feedback} is of the form~\eqref{eq:feedback_abstract_form} with
\begin{align*}
\dynavar(t)
:=
\col (\chi, \psi, \phi) \in \spaceY
\end{align*}

The corresponding state space is $
\spaceY =L^2((0,1), \R^{m+1+(d-1)m}) =L^2((0,1), \R^{dm+1}).
$
Moreover, 
\begin{equation*}
    \dynaopvar Y:= -\Lambda_{ST} \partial_x Y , \quad \forall~ Y \in D(\dynaopvar),
\end{equation*}
with $\Lambda_{ST}:= \diag\left(
\Lambda_S, S_{m+1}^{-1},  
T_2^{-1}\operatorname{Id}_m,\ldots,
T_d^{-1} \operatorname{Id}_m
\right),$ where $\Lambda_S:= \diag (S_1 ^{-1}, \cdots, S_m^{-1})$,
with domain
\begin{align*}
D(\dynaopvar)
&=
\{
\dynavar\in
H^1((0,1), \R^{dm+1})
:
\\
&\psi(0)=\mathcal L_1 Y, ~\chi(0) = 0,~ \phi(0) = 0 
\},
\end{align*}
where
$
\mathcal L_1 Y
:=
\sum_{i=1}^m
S_i\int_0^1 g_i(S_i x)\chi_i(x)\,dx
+
S_{m+1}\int_0^1 f(S_{m+1}x)\psi(x)\,dx.
$
Furthermore, the operator $\dynaopstatevar$  is written as $\dynaopstatevar = \mathcal Q\mathcal T, $ where $\mathcal Q$ is defined through its adjoint. More precisely, for every \[ Z = \operatorname{col} (z_\chi,z_\psi,z_{\phi}) \in  D(\dynaopvar^\star)\] with 
$D(\dynaopvar^\star):=
\left\{
Z
\in H^1(0,1;\mathbb R^{dm+1})
:
Z(1)=0
\right\}, $\(\mathcal Q\) is
defined through
the duality identity \begin{align*} \left\langle \mathcal Q\gamma,Z \right\rangle_{D(\dynaopvar^\star)', D(\dynaopvar^\star)} &= \left\langle \gamma,\mathcal Q^\star Z \right\rangle \\
&= \beta(1)^\star \left( \Lambda_S z_\chi(0) + \sum_{i=2}^d T_i^{-1}z_{\phi_i}(0) \right).\end{align*}
For \(\dynavar =\col(\chi,\psi,\phi)\in\spaceY\), define
\[
\mathcal M Y
:=
\operatorname{col}
\left(
\mathcal M_1Y,
\left(
T_i u_i^\star \int_0^1 \phi_i(x)\,dx
\right)_{i=2}^d
\right),
\]
where
$
\mathcal M_1Y
:=
\sum_{\ell=1}^m
S_\ell \int_0^1 g_\ell(S_\ell x)\chi_\ell(x)\,dx
+
S_{m+1}\int_0^1 f(S_{m+1}x)\psi(x)\,dx.
$
Then the dynamic feedback operator is
\[
    \feedbackdynaop
    :=
    (\operatorname{Id}_d-v)^{-1}\mathcal M,
\]
where $\operatorname{Id}_d-v$ is invertible because $v$ is strictly lower triangular.
so that $
    U(t)=\feedbackdynaop Y(t).$
\subsection{Main Result}
The following theorem is the main result of the article.
\begin{thm}
    \label{main_result} Under Assumptions~\ref{Assum_OL_stab} and \ref{assum:stabilizability_edp}, 
there exist $\dynavar$, $\dynaopstatevar, \dynaopvar, \feedbackdynaop$ (the ones given in Section~\ref{section:command_law_structure}) such that the closed-loop~\eqref{eq:hyperbolic_couple}-\eqref{eq:feedback_abstract_form} is well-posed in $L^2$ and \emph{exponentially stable} in the sense of Definition~\ref{Def_stability}.
{Moreover, if $w_0\in D(\mathcal A)$ and $\dynavar_0\in H^1((0,1), \R^{dm+1})$ is such that for all $j\in \{2,\cdots, d\}$,\begin{equation} \label{eq:compatibility_Y}\chi_0(0) =\phi_{0,j}(0)= \beta_0(1), \quad \psi_0(0) = \mathcal L_1 \dynavar_0,\end{equation}
with $\gamma_0 = (\alpha_0, \beta_0) := \mathcal T w_0$,} then $(w,\dynavar)\in C([0,\infty), H^1([0,1], \R^{n+m}))\times C([0,\infty), H^1([0,1], \R^{dm+1})).$
\end{thm}
The rest of the article is dedicated to proving Theorem~\ref{main_result}.
The methodology is the following.
\begin{enumerate}
    \item System~\eqref{eq:hyperbolic_couple} is mapped into~\eqref{eq:hyperbolic_target_2} using an invertible backstepping transform.
    \item We prove that the exponential stability in closed-loop of~\eqref{eq:hyperbolic_target_2} and~\eqref{eq:dynamical_feedback} is equivalent to the exponential stability of an IDE.
    \item The stabilizing gains~$\gains\in\gainspace$ are determined at the IDE level by applying a stable rank-reduction procedure, which reduces the problem to a single-input control problem while preserving stabilizability properties, and by subsequently solving an interpolation equation arising from a Corona problem.
\end{enumerate}
\begin{rem} \label{rem:spatially_lambda}
In the case of spatially varying transport speeds $x \mapsto \Lambda(x)$ in~\eqref{eq:hyperbolic_couple} (see, e.g.,~\cite{spatially_varying}), the method proposed in this paper remains applicable. This requires only minor adjustments to the method of characteristics used to derive the IDE, together with corresponding modifications of the backstepping kernel equations.
\end{rem}
\section{IDE Reformulation}
\label{section:IDE}
In this section, we show that the closed-loop system~\eqref{eq:hyperbolic_target_2} with the control law \eqref{eq:dynamical_feedback} can be rewritten as an IDE with equivalent stability properties.
\subsection{Stabilization problem on the IDE}
{Let $w_0\in D(\mathcal A)$ and $Y_0\in H^1((0,1),\mathbb R^{dm+1})$ satisfy the compatibility condition~\eqref{eq:compatibility_Y}. In what follows, we assume the $H^1$-well-posedness of the closed-loop system~\eqref{eq:hyperbolic_target_2}--\eqref{eq:dynamical_feedback}. In particular, its solution satisfies
$
\gamma\in C([0,\infty),H^1([0,1],\mathbb R^{n+m})).
$
This well-posedness is addressed at the beginning of Section~\ref{section:proof_main_result}.}
Set $X(t) = \beta(t,1)$ for all $t \geq 0$. We have the following lemma.
\begin{lem} \label{lem_derivation_IDE_closed_loop}
There exist parameters $l > 0$, $A_k \in \R^{m\times m}$, $\widetilde{B}_i \in \R^{m\times d}$, $\tau_k > 0$, $\theta_i > 0$ (where $k\in \{1,\cdots, l\}$ and $i\in\{1,\cdots,n\}$), $N\in L^2((0, \dmax), \R^{m\times m})$, $M\in L^2((0, \dmax), \R^{m\times d})$, and $\dmax > 0$, such that, for all $t \geq \dmax$, $(X, U)$ satisfies
\begin{equation} \label{eq:ide-input-target}\begin{aligned}
&X(t) = \sum_{k=1}^l A_k X(t-\tau_k) + \int_0^{\tau_*} N(\nu) X(t-\nu)\,d\nu \\
&+\sum_{i=1}^{n} \widetilde{B}_i U(t-\theta_i) + \int_0^{\tau_*} M(\nu) U(t-\nu)\,d\nu + B_1 U(t),
\end{aligned}
\end{equation}
for all $j\in \{2,\cdots, d\}$,
\begin{equation} \label{eq:Ui_rank_reduction}
U_j(t) = \sum_{k=1}^{j-1} v_{jk} U_k(t) + u_j^\star \int_0^{T_j} X(t-\nu)\,d\nu,
\end{equation}
and
\begin{equation} \label{eq:U1_form}
U_1(t) = \sum_{i=1}^m \int_0^{S_i} g_i(\eta) X_i(t-\eta)\,d\eta + \int_0^{S_{m+1}} f(\eta) U_1(t-\eta)\,d\eta,
\end{equation}
where $\dmax>0$ is the largest transport time, i.e.,
\begin{equation} \label{eq:retard_max}
\dmax = \max\{\tau_1,\ldots,\tau_l,\,\theta_1,\ldots,\theta_n,\,S_1,\ldots,S_{m+1},\,T_2,\ldots,T_{d}\}.
\end{equation}
\end{lem}
This result is proved in Appendix~\ref{Appendix_IDE}. The delays $\tau_k$ and $\theta_i$ depend on the transport speeds encoded in $\Lambda$ in~\eqref{eq:hyperbolic_couple}; the matrices $A_k$ and $\widetilde{B}_i$ depend on the reflection matrices $Q$ and $R$; $N$ depends on the in-domain coupling matrix $\Sigma$; and $M$ depends on the function $h$.
The initial condition of the closed-loop IDE~\eqref{eq:ide-input-target}-\eqref{eq:U1_form} is a history function $$(X_{\dmax}, U_{\dmax}) \in (L^2((0,\dmax);\R^{m}))\times (L^2((0,\dmax);\R^{d})),$$ and the state $(X_t, U_t) \in (L^2((0,\dmax);\R^{m}))\times (L^2((0,\dmax);\R^{d}))$ is defined for almost all $\eta \in [0,\dmax]$ by $(X_t(\eta), U_t(\eta)) = (X(t-\eta), U(t-\eta))$. For background on IDEs, we refer to~\cite{braun2026spectralexponentialstabilitycriterion,halebook}. We recall the following definition.
\begin{defn}[$L^2$-exponential stability (IDE)] \label{def:stability_ide}
The zero solution of the closed-loop IDE~\eqref{eq:ide-input-target}-\eqref{eq:U1_form} is exponentially stable if there exist $\kappa_0 \geq 1$ and $\omega > 0$ such that, for all $(X_{\dmax}, U_{\dmax}) \in (L^2((0,\dmax);\R^{m}))\times (L^2((0,\dmax);\R^{d}))$ and all $t \geq \dmax$,
$$\|X_t\|_{L^2} + \|U_t\|_{L^2} \leq \kappa_0 e^{-\omega t}\bigl(\|X_{\dmax}\|_{L^2} + \|U_{\dmax}\|_{L^2}\bigr).$$
\end{defn}
The IDE~\eqref{eq:ide-input-target}-\eqref{eq:U1_form} is of interest because its stability is equivalent to that of the PDE~\eqref{eq:hyperbolic_target_2} and \eqref{eq:dynamical_feedback}, as stated by the following theorem. 

\begin{thm}{\cite[Theorem 6.1.3, Chap.6]{auriol_hdr}} \label{thm:equiv_stab_ide_edp}
Let $\gains \in \gainspace$. The closed-loop system ~\eqref{eq:hyperbolic_target_2} with the control law~\eqref{eq:dynamical_feedback} is exponentially stable if and only if the IDE closed-loop~\eqref{eq:ide-input-target}-\eqref{eq:U1_form} is exponentially stable.
\end{thm}
\begin{proof}
  For the sake of completeness, the proof of Theorem~\ref{thm:equiv_stab_ide_edp} is given in the complementary material files available to the reviewers.   
\end{proof}
 We therefore seek gains $\gains \in \gainspace$ such that the IDE closed-loop~\eqref{eq:ide-input-target}-\eqref{eq:U1_form} is exponentially stable in the sense of Definition~\ref{def:stability_ide}.
\subsection{Stability Properties of the IDE}
Before designing the controller gains that guarantee the exponential stability of the IDE system~\eqref{eq:ide-input-target}-\eqref{eq:U1_form}, we first establish several stability-related properties of the IDE~\eqref{eq:ide-input-target}-\eqref{eq:U1_form}. In particular, the proposed stability analysis will be done in the Laplace domain using  the characteristic equation~\cite{halebook} associated with the closed-loop dynamics. Therefore, the contribution of the initial conditions can be disregarded as they do not affect the location of the characteristic roots governing stability. Formally taking the Laplace transform of~\eqref{eq:ide-input-target}, we obtain for all $\omega>0$ and all $s\in \C_\omega$

\begin{equation}
    \label{eq:ide_laplace}
    \qhat(s) X(s) = \phat_1(s)U_1(s)+\ldots+\phat_d(s) U_d(s).
\end{equation}
where, 
\begin{equation}
\label{eq:Qs-def-proof}
\qhat(s):=\Delta_0(s)-\Nhat(s) \in \C^{m\times m}
\end{equation}
with 
\begin{equation} \label{def:pp}\Delta_0(s):=\operatorname{Id}_m-\sum_{k=1}^l A_k e^{-\tau_k s}.\end{equation}
and
\begin{equation}
\begin{aligned}
\label{eq:Ps-def-proof}
\phat(s)&=[\phat_1(s),\cdots, \phat_d(s)]\\
:&=B_1+\sum_{i=1}^n \widetilde B_i e^{-s\theta_i}
+\Mhat(s) \in \C^{m\times d}
\end{aligned}
\end{equation}
We have the following consequence of Assumption~\ref{Assum_OL_stab}.
\begin{corollary} \label{coro:pp_ide_stable}
    Under Assumption~\ref{Assum_OL_stab}, there exists $\omega_1>0$ and $\infpp>0$ such that, for all $s\in \C_{\omega_1}$,
    \begin{equation}
        \label{inf:pp_positif}
       | \det \Delta_0(s)|>\infpp.
    \end{equation}
\end{corollary}
\begin{proof}
   First, observe that applying the same computations as in the proof of Lemma~\ref{lem_derivation_IDE_closed_loop} to the open-loop system~\eqref{eq:hyperbolic_couple}, with $\Sigma=0$ and $U\equiv 0$, yields the principal part of the IDE~\eqref{eq:ide-input-target}, namely
    \begin{equation}
        \label{eq:ide_pp}
        X(t) =
\sum_{k=1}^l A_kX(t-\tau_k),
    \end{equation}
where $A_k$ and $\tau_k$ are defined in~\eqref{eq:Ak-input}.
By adapting the proof of Theorem~\ref{thm:equiv_stab_ide_edp} by choosing the control law $U\equiv 0$ and the in-domain coupling terms $\Sigma \equiv 0$ in~\eqref{eq:hyperbolic_couple}, Assumption~\ref{Assum_OL_stab} implies that the principal part of the IDE~\eqref{eq:ide-input-target} is exponentially stable in the sense of Definition~\ref{def:stability_ide}. By~\cite[Theorem 6]{braun2026spectralexponentialstabilitycriterion}, this is equivalent to the existence of $\omega_0>0$ such that $\det \Delta_0(s) \neq 0$ for all $s\in \C_{\omega_0}$.
Finally, using~\cite[Proposition A.5]{felipe}, there exist $\omega_1< \omega_0$ and $\infpp>0$ such that~\eqref{inf:pp_positif} holds.
\end{proof}
Next, we show that Assumption~\ref{assum:stabilizability_edp} implies a stabilizability condition on the IDE~\eqref{eq:ide-input-target}, namely a maximal rank condition. This condition is believed to be  a strong stabilizability condition, consistent with the case without distributed delays (see~\cite[Theorem 3.1]{hale_feedback_2} and the discussion after equation (3.3)
in the same reference).
\begin{prop} \label{prop:full_rank_ide}
Under Assumption~\ref{assum:stabilizability_edp} and for the same $\omega >0$ as in Assumption~\ref{assum:stabilizability_edp},
\begin{equation}
\label{eq:ide-rank-ass}
\operatorname{rank}[\qhat(s),-\phat(s)]=m,
\qquad \forall s\in\mathbb C_{\omega},
\end{equation}
\end{prop}
The proof of Proposition~\ref{prop:full_rank_ide} is postponed to Appendix~\ref{section:preuve_prop_full_rank}.
Let \begin{equation} \label{eq:defomega2}\omega_2 < \min(\omega_1, \omega)\end{equation} with $\omega$ from Proposition~\ref{prop:full_rank_ide} and $\omega_1$ from Corollary~\ref{coro:pp_ide_stable}. We now have all the tools needed to design the controller gains.

\section{Design of the control Law}
\label{section:design_command_law}
\subsection{Reduction to a Single Control Input}
The aim of this subsection is to select the gains
$v \in \mathbb{R}^{d\times d}$ lower triangular, $T_2,\dots,T_d > 0$, and $u \in \R^{m\times d-1}$
such that the controls $U_2,\ldots,U_{d}$ defined in~\eqref{eq:Ui_rank_reduction}
yield an IDE involving only $X$ and $U_1$, while satisfying a stabilizability
rank condition analogous to~\eqref{eq:ide-rank-ass}.
The methodology extends the approach of~\cite{braun_stab_3_2inputs}, where the
stabilization of a scalar IDE with two inputs was reduced to that of a one-input
IDE satisfying a maximal rank condition; see also~\cite{quadrat2004general} for
related algebraic rank conditions over rings of delay operators.
For any $T>0$, define for all $s\in \mathbb C$
\begin{equation}\label{def:K_T}
K_T(s)
=
\int_0^T e^{-ts}\,dt
=
\frac{1-e^{-Ts}}{s},
\qquad K_T(0)=T.
\end{equation}
\begin{thm}\label{thm:rank_reduction}
Suppose that Assumptions~\ref{Assum_OL_stab} and~\ref{assum:stabilizability_edp}
hold. Then, for every $\varepsilon>0$, there exists a countable nowhere dense set
$\mathcal{T}_{\mathrm{bad}}\subset\mathbb{R}_+$ such that, for every
$T\in(0,\varepsilon)\setminus\mathcal{T}_{\mathrm{bad}}$, there exists a set
$\mathcal{F}_{\mathrm{bad}}\subset\mathbb{R}^m$ contained in a finite union of
proper hyperplanes of $\mathbb{R}^m$, such that the following holds.
For every $u\in\mathbb{R}^m\setminus\mathcal{F}_{\mathrm{bad}}$, there exists a
finite set $\mathcal{L}_{\mathrm{bad}}(u)\subset\mathbb{R}^{d-1}$ such that, for
every $v\in\mathbb{R}^{d-1}\setminus\mathcal{L}_{\mathrm{bad}}(u)$,
\[
\rank M_{u,v,T}(s)=m, \qquad \forall\,s\in\mathbb{C}_{\omega_2},
\]
where
$
M_{u,v,T}(s) 
= \Bigl[\,\qhat(s)-K_T(s)\phat_d(s)u^\star,\;
\phat_1(s)-v_1\phat_d(s),\;\ldots,\;
\phat_{d-1}(s)-v_{d-1}\phat_d(s)\Bigr].
$
In particular, $T$ may be chosen outside a countable non-dense set, then $u$
outside a finite union of hyperplanes, and then $v$ outside a finite set
depending on $u$ (see Algorithm~\ref{algo:gains_rank_reduction} for practical guidance).
\end{thm}
\begin{proof}
    The proof is postponed to Appendix~\ref{section:proof_thm_stable_rank}.
\end{proof}
\begin{algorithm}[H]
\caption{Choosing the gains for $U_d, \dots , U_2$} \label{algo:gains_rank_reduction}
To find the gains in~\eqref{eq:Ui_rank_reduction}, proceed as follows, for
$i=d,d-1,\ldots,2$:
\begin{enumerate}
  \item Choose a small number $T_i \in (0,\varepsilon)$, out of the negligible set $\mathcal{T}_{\mathrm{bad}}$.
  \item Choose $u_i\in\mathbb{R}^m$, out of the negligible set $\mathcal{F}_{\mathrm{bad}}$
  \item Choose $v_i=(v_{ij})\in\mathbb{R}^{i-1}$, out of the finite set
        $\mathcal{L}_{\mathrm{bad}}(u_i)$.
\end{enumerate}
\end{algorithm}
Iterating Theorem~\ref{thm:rank_reduction} $d-1$ times yields a single-input IDE
in which the only remaining free control is $U_1$:
\begin{equation}\label{eq:ide_single_input}
\begin{aligned}
&X(t) = \sum_{k=1}^l A_k X(t-\tau_k)
+ \int_0^{\tau} N_f(\nu)X(t-\nu)\,d\nu \\
&+ \sum_{i=1}^I \bar{B}_i U_1(t-\sigma_i)
+ \int_0^{\tau} M_f(\nu)U_1(t-\nu)\,d\nu
+ \bar{B}_0 U_1(t),
\end{aligned}
\end{equation}
where $N_f$ and $M_f$ are $L^2$ kernels, $\tau>\tau_*$,
$\bar{B}_i\in\mathbb{R}^m$, $\sigma_i>0$, and $I>n$.
Note that $K_T$ in~\eqref{def:K_T} does not alter the principal part of the IDE,
so the matrices $A_k$ and delays $\tau_k$ are unchanged from~\eqref{eq:ide-input-target}
and Corollary~\ref{coro:pp_ide_stable} remains applicable.
Moreover, the maximal rank condition with a single input holds (see Remark~\ref{rem:valid_rank_reduction}):
\begin{equation}\label{eq:max_rank_final}
\rank[\,\qhat_f(s),\,-\phat_f(s)\,] = m, \qquad \forall\,s\in\mathbb{C}_{\omega_2},
\end{equation}
where $\omega_2$ is defined in~\eqref{eq:defomega2}, 
\begin{equation*}
\qhat_f(s) := \Delta_0(s) - \Nhat_f(s) \in \C^{m\times m},
\end{equation*}
\begin{equation*}
\phat_f(s) := \bar{B}_0 + \sum_{i=1}^I \bar{B}_i\,e^{-s\sigma_i} + \Mhat_f(s)\in \C^{m}.
\end{equation*}
{At this stage, we have fixed the components $u,v,T$ of the gains $\gains \in \gainspace$ of the dynamical feedback~\eqref{eq:dynamical_feedback}.}

\subsection{Design of the gains for \texorpdfstring{$U_1$}{U1}: a Corona-based approach}
This subsection is devoted to the construction of the gains $g$ and $f$ such that the closed-loop system formed by~\eqref{eq:U1_form} and~\eqref{eq:ide_single_input} is exponentially stable. This problem was recently addressed in~\cite{braun_corona}; for completeness, we recall here the main steps of the methodology.
\begin{thm}\label{thm:gain_Corona_main} Under Assumptions~\ref{Assum_OL_stab} and~\ref{assum:stabilizability_edp}, there exist controller gains $$(g,f)\in \prod_{i=1}^m L^2\bigl((0,S_i),\mathbb{R}\bigr)\times L^2\bigl((0,S_{m+1}),\R\bigl)$$ such that the IDE closed-loop formed by~\eqref{eq:U1_form} and \eqref{eq:ide_single_input} is exponentially stable in the sense of Definition~\ref{def:stability_ide}.
\end{thm}
\begin{proof}
    The complete proof for $m\in \N^\star$ is given in~\cite[Section~V]{braun_corona}; we restate the main argument here for $m=1$.
    Taking the Laplace transform of~\eqref{eq:U1_form}-\eqref{eq:ide_single_input}
yields
\begin{equation*}
\begin{bmatrix}
\widehat{q}_f(s) & -\widehat{p}_f(s)\\
-\widehat{g}(s) & 1-\widehat{f}(s)
\end{bmatrix}
\begin{bmatrix}\widehat{X}(s)\\ \widehat{U}_1(s)\end{bmatrix}=0,
\end{equation*}
hence the associated characteristic equation is
\begin{align}
0 &= \det\begin{bmatrix}
\widehat{q}_f(s) & -\widehat{p}_f(s)\\
-\widehat{g}(s) & 1-\widehat{f}(s)
\end{bmatrix} \label{eq:ide_characteristic}\\
&= 1 - \sum_{k=1}^l A_k e^{-\tau_k s}
- \widehat{N}_f(s)
- \widehat{q}_f(s)\widehat{f}(s)
- \widehat{p}_f(s)\widehat{g}(s). \nonumber
\end{align}
We use the following spectral characterization.
\begin{lem}[{\cite[Theorem~6]{braun2026spectralexponentialstabilitycriterion}}]
\label{stable_lemma}
The closed-loop~\eqref{eq:U1_form}-\eqref{eq:ide_single_input} is exponentially
stable if and only if there exists $\omega>0$ such that~\eqref{eq:ide_characteristic}
has no solution in $\mathbb{C}_\omega$.
\end{lem}
By Corollary~\ref{coro:pp_ide_stable} and Lemma~\ref{stable_lemma}, it suffices
to find $g$ and $f$ such that
\begin{equation}
\label{eq:Corona_approx}
\bigl|\widehat{q}_f(s)\widehat{f}(s)
+ \widehat{p}_f(s)\widehat{g}(s)
+ \widehat{N}_f(s)\bigr| < \infpp.
\end{equation}
This is an approximate Corona problem~\cite[Chapter~12]{Duren2000Hp};
the standard Corona problem corresponds to~\eqref{eq:Corona_approx} with $\infpp=0$ {and the inequality replaced by an equality.}
By~\cite[Theorem~8]{braun_corona}, there exists a unique minimal-norm solution
\[
(g^\star,f^\star)\in L^2_{\omega_2}\bigl((0,\infty),\mathbb{R}^2\bigr),
\qquad 0<\omega_2<\omega_1,
\]
of~\eqref{eq:Corona_approx} with $\infpp=0$.
Since for all $0<\omega_3<\omega_2$, all
$v\in L^2_{\omega_2}((0,\infty),\mathbb{C})$, and all $S\geq 0$,
\begin{equation*}
\|v\|_{L^2_{\omega_3}(S,\infty)}
\leq e^{-(\omega_2-\omega_3)S}\|v\|_{L^2_{\omega_2}},
\end{equation*}
the tails of $(g^\star,f^\star)$ can be truncated to obtain compactly supported gains.
Specifically, a solution of~\eqref{eq:Corona_approx} is given by
\[
g = g^\star\,\mathbf{1}_{[0,S_1]}, \qquad f = f^\star\,\mathbf{1}_{[0,S_2]},
\]
where $S_1$ and $S_2$ satisfy the lower bound of~\cite[Theorem~8]{braun_corona}.
\end{proof}

\subsection{Proof of the Main Result}
\label{section:proof_main_result}
\begin{proof}
We now prove Theorem~\ref{main_result}. Although the closed-loop formed by~\eqref{eq:hyperbolic_couple} and \eqref{eq:dynamical_feedback} does not fall exactly within the framework considered in~\cite[Appendix A]{bastin_coron}, owing to the presence of additional integral terms, the analytical arguments developed therein can be adapted to the present setting. In particular, the well-posedness analysis may be carried out by following the same semigroup-based approach, relying on the Lumer-Phillips theorem~\cite{lumer1961dissipative}.
By the invertibility of the backstepping transformation~$\mathcal{T}$ established in Lemma~\ref{lem:backsteppin_transform}, the exponential stability of the original closed-loop formed by~\eqref{eq:hyperbolic_couple} and \eqref{eq:dynamical_feedback} is equivalent to exponential stability of the target system~\eqref{eq:hyperbolic_target_2} and~\eqref{eq:dynamical_feedback}. By Theorem~\ref{thm:equiv_stab_ide_edp}, the latter is in turn equivalent to exponential stability of the IDE closed-loop~\eqref{eq:ide-input-target}-\eqref{eq:U1_form}. 
Following Algorithm~\ref{algo:gains_rank_reduction}, we obtain gains $v\in\mathbb{R}^{(d-1)\times(d-1)}$, $u\in\mathbb{R}^{m\times(d-1)}$, $T_i>0$, $S_i>0$, yielding the single-input IDE~\eqref{eq:ide_single_input} together with the maximal rank condition~\eqref{eq:max_rank_final}. Finally, Theorem~\ref{thm:gain_Corona_main} guarantees the existence of $L^2$ gains $g,f$ that render the  closed-loop IDE formed by~\eqref{eq:U1_form} and \eqref{eq:ide_single_input} exponentially stable. For such gains, using Theorem~\ref{thm:equiv_stab_ide_edp}, the original closed-loop formed by~\eqref{eq:hyperbolic_couple} and \eqref{eq:dynamical_feedback} is therefore exponentially stable.
\end{proof}
\section{Application to Network Stabilization}
\label{section:network}
In this section, we show how the proposed method applies to the stabilization of networks of first-order hyperbolic PDEs. The construction relies on a pairwise consistency assumption, which defines a signed graph associated with the network. Cutting and folding transformations then recast the original network as a PDE system of the form~\eqref{eq:hyperbolic_couple}. Further details about signed graphs can be found in~\cite{harary1953notion_graph,algo_balance_graph}. 
\subsection{Port-Networks of PDEs}

We introduce the concept of Port-Networks 
to obtain a unified framework for the stabilization of networks of hyperbolic PDEs~\cite{redaud2024output,auriol_hdr,braun_stab_3_2inputs}. We consider a collection of $n_g \geq 2$ subsystems of first-order hyperbolic PDEs coupled at the boundary.
{Each subsystem takes the form of~\eqref{eq:hyperbolic_couple}, i.e.} for all $t>0$ and all $x\in [0,1]$, 
\begin{equation} \label{eq:system_i_graph}
\partial_t w_i(t,x) + \Lambda_i\,\partial_x w_i(t,x)
= \Sigma_i(x)\,w_i(t,x) + h_i(x)\,U(t),
\end{equation}
where $w_i = \col(w_i^+,w_i^-)\in\mathbb{R}^{n_i+m_i}$ $((n_i,m_i)\in \mathbb{N}\backslash\{0\})$ and all parameters are
defined analogously to~\eqref{eq:hyperbolic_couple}.
{Each subsystem may contain boundary internal couplings of the same form as the boundary couplings $Q$ and $R$ in~\eqref{eq:hyperbolic_couple}. Since the subsystems of the network are interconnected only through their boundaries, we will represent these internal couplings by the diagonal terms of the network interconnection matrix. The boundary interconnections are formalized by assigning to each subsystem two ports, $p_i^{+1}$ and $p_i^{-1}$, corresponding respectively to $x=0$ and $x=1$; see Figure~\ref{fig:generic_vertex_hyperbolic}.}
At port $p_i^{+1}$, the incoming and outgoing variables are respectively
\[
r_i^{+1} := w_i^+(t,0)\in\mathbb{R}^{n_i},
\qquad
y_i^{+1} := w_i^-(t,0)\in\mathbb{R}^{m_i}.
\]
At port $p_i^{-1}$, they are
\[
r_i^{-1} := w_i^-(t,1)\in\mathbb{R}^{m_i},
\qquad
y_i^{-1} := w_i^+(t,1)\in\mathbb{R}^{n_i}.
\]
Define the aggregated incoming and outgoing vectors
\[
r^{\pm 1} := \col(r_1^{\pm 1},\ldots,r_{n_g}^{\pm 1}),
\qquad
y^{\pm 1} := \col(y_1^{\pm 1},\ldots,y_{n_g}^{\pm 1}),
\]
and set $r := \col(r^{+1},r^{-1})$, $y := \col(y^{+1},y^{-1})$.
The boundary interconnection of the network is written as
\begin{equation} \label{eq:boundary_couplings_graph}
r = K\,y + \Bgraph\,U(t), \quad \Bgraph\in\mathbb{R}^{\bigl(\sum_{i=1}^{n_g}(n_i+m_i)\bigr)\times d}
\end{equation}
or equivalently
\[
\begin{pmatrix} r^{+1}\\ r^{-1} \end{pmatrix}
=
\begin{pmatrix} K_{++} & K_{+-}\\ K_{-+} & K_{--} \end{pmatrix}
\begin{pmatrix} y^{+1}\\ y^{-1} \end{pmatrix}
+
\begin{pmatrix} B_{\mathrm{g},0}\\ B_{\mathrm{g},1} \end{pmatrix}U(t).
\]
The matrix $K$ is block-structured as
\[
K_{lk} = \bigl[K_{(i,l),(j,k)}\bigr]_{i,j\in\{1,\ldots,n_g\}},
\qquad (l,k)\in\{+1,-1\}^2,
\]
where $K_{(i,l),(j,k)}$ maps the outgoing signal $y_j^k$ at port $p_j^k$ to
the incoming signal $r_i^l$ at port $p_i^l$.
Cross-port self-couplings are excluded since every subsystem is assumed to be already in the canonical form~\eqref{eq:hyperbolic_couple} for the folding method to apply, i.e. for all $i\in\{1,\ldots,n_g\}$,
\begin{equation} \label{eq:self_coupling_constraint}
K_{(i,+1),(i,-1)} = 0,
\qquad
K_{(i,-1),(i,+1)} = 0.
\end{equation}
\begin{figure}[htb]
\centering
\scalebox{0.65}{
\begin{tikzpicture}[>=Stealth]

  \PortVertex{0}{0}{i}{1}

  \draw[<->,very thick] (1.0,0) -- (2.4,0);
  \node[above] at (1.7,0) {\small equivalent};

  \begin{scope}[shift={(5.0,-0.1)}]
    \GenericHyperbolicVertex{(0,0)}
  \end{scope}

\end{tikzpicture}}
\caption{Equivalence between a port-network vertex \(i\) and a hyperbolic subsystem.}
\label{fig:generic_vertex_hyperbolic}
\end{figure}
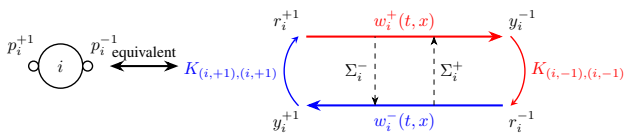

\begin{defn}[Port-Network]
A port-network is a collection of $n_g\geq 2$ subsystems~\eqref{eq:system_i_graph}
coupled at the boundary via~\eqref{eq:boundary_couplings_graph} and
subject to~\eqref{eq:self_coupling_constraint}. 
\end{defn}

\subsection{Network Objective and Associated Signed Graph}
The objective of this section is to transform a port-network into a system like~\eqref{eq:hyperbolic_couple}. More precisely, we seek a cutting/folding transformation $\mathcal{F}$ (see Figures~\ref{fig:folding_transformation}-\ref{fig:cutting_transformation}), in the spirit of~\cite{folding_vazquez, auriol2022folding}, such that the transformed port network has boundary conditions of the form~\eqref{eq:boundary_couplings_graph} with a boundary coupling matrix that is block-diagonal by port type, i.e.,
\[
K_{+-} = 0, \qquad K_{-+} = 0.
\]
The transformation $\mathcal{F}$ is required to be linear, bounded, and boundedly invertible on the state space.
{This establishes the equivalence between the original and transformed networks at the level of abstract evolution systems. In particular, the folding and cutting transformations preserve both well-posedness and stabilizability. Therefore, stabilizability assumptions such as Assumption~\ref{Assum_OL_stab}, together with robustness assumptions such as Assumption~\ref{assum:stabilizability_edp}, remain valid for the folded system whenever they hold for the original port network.}
The transformation $\mathcal F$ can be constructed under a pairwise consistency condition on the boundary couplings. This condition allows one to associate a well-defined signed graph with the port network, whose edge labels encode whether two coupled ports must be folded in the same or in the opposite direction.

\begin{defn}[Signatures] \label{def:signature}
Whenever $K_{(i,\ell),(j,k)}\neq 0$, we define an undirected constraint edge between nodes $i$ and $j$.
The \emph{signature} of this edge is
\[
\sigma_{\ell k}(i,j) := \ell \times k \in \{-1,+1\}.
\]
The signature of a path or a cycle is the product of the edge signatures along it.
\end{defn}
Consequently, $\sigma_{\ell k}(i,j)=+1$ means the coupling connects ports {on the same boundary} 
\[
p_i^{+1}\leftrightarrow p_j^{+1}
\quad\text{or}\quad
p_i^{-1}\leftrightarrow p_j^{-1},
\]
whereas $\sigma_{\ell k}(i,j)=-1$ means it connects ports of opposite types,
\[
p_i^{+1}\leftrightarrow p_j^{-1}
\quad\text{or}\quad
p_i^{-1}\leftrightarrow p_j^{+1}.
\]
If several nonzero blocks couple the same pair $(i,j)$, they may \textit{a priori}
induce different signatures.
The following condition excludes such contradictions.
\begin{assum}[Pairwise consistency]\label{assump:pairwise_consistency}
For every pair $(i,j)\in\{1,\ldots,n_g\}^2$, all nonzero boundary couplings
between subsystems $i$ and $j$ induce the same signature.
That is, for all $(\ell,k),(\ell',k')\in\{-1,+1\}^2$, if
$K_{(i,\ell),(j,k)}\neq 0$ and $K_{(i,\ell'),(j,k')}\neq 0$, then
\[
\sigma_{\ell k}(i,j) = \sigma_{\ell' k'}(i,j).
\]
We denote this common value by $\sigma(i,j)\in\{-1,+1\}$.
\end{assum}
{As an example, the port-network represented in Figure~\ref{fig:pairwise_consistency_counterexample_two_vertices} does not satisfy Assumption~\ref{assump:pairwise_consistency}.}
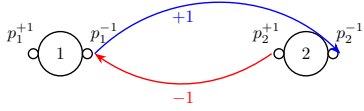
\begin{figure}[htb]
\centering
\scalebox{0.65}{
\begin{tikzpicture}[>=Stealth]

  \PortVertex{0}{0}{1}{1}
  \PortVertex{5}{0}{2}{1}

  \coordinate (p1p) at (-0.70,0);
  \coordinate (p1m) at ( 0.70,0);

  \coordinate (p2p) at ($(5,0)+(-0.70,0)$);
  \coordinate (p2m) at ($(5,0)+( 0.70,0)$);

  \draw[thick,->,red]
  (p2p) to[bend left=35]
  (p1m);

  \node[red] at (2.5,-0.9) {$-1$};

  \draw[thick,->,blue]
  (p1m) to[bend left=45]
  (p2m);

  \node[blue] at (2.5,0.75) {$+1$};

\end{tikzpicture}}
\caption{Two couplings between subsystems \(1\) and \(2\) inducing different signatures.}
\label{fig:pairwise_consistency_counterexample_two_vertices}
\end{figure}
In particular, constraint~\eqref{eq:self_coupling_constraint} implies that
self-couplings always have signature $+1$. Under Assumption~\ref{assump:pairwise_consistency}, one can represent a port-network with undirected edges like in Figure~\ref{fig:port_network_considered}.
\begin{figure}[htb]
\centering
\scalebox{0.7}{
\begin{tikzpicture}[>=Stealth]

  \coordinate (v1) at (0,0);
  \coordinate (v2) at (4,0);
  \coordinate (v3) at (4,3);
  \coordinate (v4) at (0,3);

  \coordinate (v1L) at ($(v1)+(-0.70,0)$);
  \coordinate (v1R) at ($(v1)+( 0.70,0)$);

  \coordinate (v2L) at ($(v2)+(-0.70,0)$);
  \coordinate (v2R) at ($(v2)+( 0.70,0)$);

  \coordinate (v3L) at ($(v3)+(-0.70,0)$);
  \coordinate (v3R) at ($(v3)+( 0.70,0)$);

  \coordinate (v4L) at ($(v4)+(-0.70,0)$);
  \coordinate (v4R) at ($(v4)+( 0.70,0)$);

\draw[thick]
  ($(v1R)+(-0.1,0)$) -- ($(v2L)+(0.1,0)$);

\draw[thick]
  ($(v2R)+(-0.1,0)$) -- ($(v3L)+(0.1,0)$);

\draw[thick]
  ($(v3R)+(-0.1,0)$) -- ($(v4L)+(0.1,0)$);

\draw[thick]
  ($(v4R)+(-0.1,0)$) -- ($(v1L)+(0.1,0)$);

  \PortVertex{0}{0}{1}{1}
  \PortVertex{4}{0}{2}{1}
  \PortVertex{4}{3}{3}{1}
  \PortVertex{0}{3}{4}{1}

  \draw[->,very thick] ($(v1L)+(-0.5,0)$) -- (v1L);
  \node[left] at ($(v1L)+(-0.5,0)$) {$U_1(t)$};
    \draw[->,very thick] ($(v3L)+(0.7,1)$) -- ($(v3L)+(0.7,0.5)$);
  \node[above] at ($(v3L)+(0.7,1)$) {$U_2(t)$};
\end{tikzpicture}}
\caption{Cycle port-network with one boundary input and one in-domain input.}
\label{fig:port_network_considered}
\end{figure}
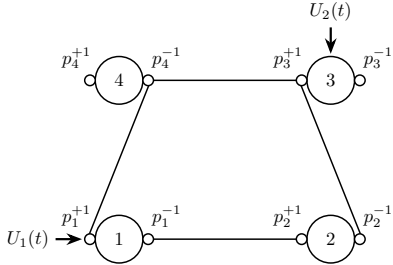
\begin{rem} \label{rem:signature1}
    A port-network where all signatures are $+1$ is already in the same form as~\eqref{eq:hyperbolic_couple}.
\end{rem}
Under Assumption~\ref{assump:pairwise_consistency}, we can associate a signed graph
with every port-network.
\begin{defn}[Signed graph of a port-network] \label{def:sign_graph}
The signed graph associated with a port-network of the
form~\eqref{eq:system_i_graph}-\eqref{eq:boundary_couplings_graph} is the
triple $\graph=(\mathcal{V},\mathcal{E},\sigma)$ where:
\begin{itemize}
  \item $\mathcal{V} := \{1,\ldots,n_g\}$ is the set of vertices, each representing
        a two-port hyperbolic subsystem of the form~\eqref{eq:system_i_graph}.
  \item $\mathcal{E}$ is the set of edges, where an edge $(i,j)$ is present
        whenever at least one nonzero coupling exists between subsystems $i$ and $j$.
        The graph is simple: Assumption~\ref{assump:pairwise_consistency} ensures
        that at most one edge connects any two vertices.
  \item $\sigma:\mathcal{E}\to\{-1,+1\}$ is the signature function, well defined
        by Assumption~\ref{assump:pairwise_consistency}.
\end{itemize}
See Figure~\ref{fig:associated_sign_graph} for an example.
\end{defn}
\begin{figure}[htb]
\centering
\scalebox{0.7}{
\begin{tikzpicture}[>=Stealth,
  vertex/.style={circle,draw,thick,minimum size=7mm,inner sep=0pt}]

  \node[vertex] (v1) at (0,0) {1};
  \node[vertex] (v2) at (4,0) {2};
  \node[vertex] (v3) at (3.2,3) {3};
  \node[vertex] (v4) at (0.8,3) {4};

  \draw[thick] (v1) -- node[below] {$-1$} (v2);
  \draw[thick] (v2) -- node[right] {$-1$} (v3);
  \draw[thick] (v3) -- node[above] {$-1$} (v4);
  \draw[thick] (v4) -- node[left] {$-1$} (v1);


\end{tikzpicture}}
\caption{Signed graph associated with the cycle network represented in Figure~\ref{fig:port_network_considered}.}
\label{fig:associated_sign_graph}
\end{figure}
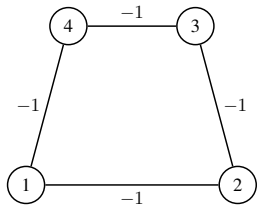
The following theorem is the main result of the section.
\begin{thm}\label{thm:graph_main}
Under Assumption~\ref{assump:pairwise_consistency}, every port
network~\eqref{eq:system_i_graph}-\eqref{eq:boundary_couplings_graph} can be
recast as a coupled hyperbolic system of the form~\eqref{eq:hyperbolic_couple}.
More precisely, letting $\mathcal{X}_g := \prod_{i=1}^{n_g}(L^2([0,1],\mathbb{R}))^{n_i+m_i}$
denote the state space of the port network, there exists a bounded linear
isomorphism
\[
\mathcal{F}:\mathcal{X}_g\to\mathcal{X}
\]
which transforms~\eqref{eq:system_i_graph}-\eqref{eq:boundary_couplings_graph}
into a system like~\eqref{eq:hyperbolic_couple}.
\end{thm}
The proof of Theorem~\ref{thm:graph_main} is given in Appendix~\ref{section:proof_thm_graph_main} and we now introduce the main tools for this proof.
\subsection{Folding Transform}
\label{section:folding}
The folding transforms are formally introduced in the proof of Theorem~\ref{thm:graph_main} in Appendix~\ref{section:proof_thm_graph_main}, for a visual explanation we refer to Figure~\ref{fig:folding_transformation}. For more explanations about folding transforms we refer to~\cite{ folding_vazquez}. 
This section is devoted to the choice of which network subsystems should be folded.
\begin{defn}[Folding choice] \label{def:folding_choice}
A folding choice is a marking (see~\cite{algo_balance_graph}) of the graph $\mathcal{G}$, i.e., a vector
\[
(m(1),\dots,m(n_g)) \in \{-1,+1\}^{n_g}.
\]
The value $m(i)=+1$ means that subsystem $i$ is not folded, whereas $m(i)=-1$
means that subsystem $i$ is folded.
\end{defn}

If subsystem $i$ is folded, define
\[
\widetilde{w}_i^+(t,x) := w_i^-(t,1-x), \qquad
\widetilde{w}_i^-(t,x) := w_i^+(t,1-x).
\]
Folding swaps the two ports of subsystem $i$.
If it is not folded, set $\widetilde{w}_i^\pm := w_i^\pm$.
After the folding transformation associated with marking $m$, the new port
label of $p_i^\ell$ becomes $\ell \times m(i)$, see Figure~\ref{fig:folding_transformation}.

\begin{defn}[Good folding]\label{def:good_folding}
A folding choice $m$ is called \emph{good} if every nonzero boundary coupling
connects two ports with the same final port label, i.e., for all
$(i,j)\in\{1,\ldots,n_g\}^2$ and all $(\ell,k)\in\{-1,+1\}^2$,
\[
K_{(i,\ell),(j,k)}\neq 0
\quad\Longrightarrow\quad
\ell\times m(i) = k\times m(j),
\]
or equivalently, $m(i)\times m(j) = \ell\times k$.
Under Assumption~\ref{assump:pairwise_consistency}, this condition reduces to
\[
m(i)\times m(j) = \sigma(i,j)
\]
for every edge $(i,j)\in\mathcal{E}$.
\end{defn}
{The purpose of the folding procedure is to ensure that all coupling signatures in the transformed network are equal to $+1$; see Remark~\ref{rem:signature1}.}
\begin{defn}[Balanced signed graph]\label{def:balanced_graph}
A signed graph $\mathcal{G}=(\mathcal{V},\mathcal{E},\sigma)$ is
\emph{balanced} if every cycle has positive signature.
That is, for every cycle
\[
i_0 | i_1 | \cdots | i_q = i_0,
\]
one has
\[
\sigma(i_0,i_1)\times\sigma(i_1,i_2)\times\cdots\times\sigma(i_{q-1},i_q) = +1.
\]
Equivalently, every cycle contains an even number of edges of signature $-1$.
\end{defn}
For example, the signed graph represented in Figure~\ref{fig:associated_sign_graph} is balanced.
\begin{lem}[Existence of a good folding choice]\label{lem:good_folding}
Assume the graph $\graph$ is balanced. Then, under Assumption~\ref{assump:pairwise_consistency}, there exist exactly $2^c$ (where $c$ is the number of connected components of the signed graph $\graph$) good folding
choices $m\in\{-1,+1\}^{n_g}$ in the sense of Definition~\ref{def:good_folding}.
\end{lem}

\begin{algorithm}[H]\caption{Construction of a Good Folding Choice}
\label{algo:folding}
A good folding vector $m$ can be constructed as follows.
For each connected component of $\mathcal{G}$:
\begin{enumerate}
  \item Choose a reference node $i_0$ and set $m(i_0)=+1$.
  \item For any other node $i$, choose a path
        $i_0 | i_1 | \cdots | i_q = i$ and define
        \[
        m(i) := \sigma(i_0,i_1)\times\sigma(i_1,i_2)\times\cdots\times\sigma(i_{q-1},i_q).
        \]
\end{enumerate}
\end{algorithm}
Hence $m(i)$ encodes the parity of the number of edges of signature $-1$ along
any path from the reference node $i_0$ to $i$; the balancedness of $\mathcal{G}$
guarantees that this value is independent of the chosen path.

\begin{proof}
The result is a direct consequence of~\cite[Theorem~1]{algo_balance_graph},
upon observing that a good folding choice $m$ exists if and only if
$\mathcal{G}$ is the signed graph induced by the vertex marking $m$, i.e.,
the signature of each edge equals the product of the markers of its endpoints.

\end{proof}
For example, the only two good folding choices for the signed graph represented in Figure~\ref{fig:associated_sign_graph} are
$(m(1), m(2),m(3), m(4)) = (+1, -1, +1, -1),$

$(m(1), m(2),m(3), m(4)) = (-1, +1, -1, +1).$
\begin{figure}[H]
    \centering
    \scalebox{0.6}{%
    \begin{tikzpicture}[>=stealth]

\node[font=\large] at (4,1.7) {Before folding};

\draw[->,red,very thick] (0,0) -- (3,0);
\node[red,above] at (1.5,0) {$w_i^+(t,x)$};

\draw[<-,blue,very thick] (0,-1.5) -- (3,-1.5);
\node[blue,below] at (1.5,-1.5) {$w_i^-(t,x)$};

\draw[blue,thick]
    (-0.5,-0.75) arc (-180:-135:1.1);
\draw[blue,->,thick]
    (-0.5,-0.75) arc (-180:-225:1.1);
\node[blue,right] at (-0.5,-0.75) {$Q_{ii}$};

\draw[red,thick]
    (3.5,-0.75) arc (0:45:1.1);
\draw[red,->,thick]
    (3.5,-0.75) arc (0:-45:1.1);
\node[red,left] at (3.5,-0.75) {$R_{ii}$};

\draw[->,red,very thick] (5,0) -- (8,0);
\node[red,above] at (6.5,0) {$w_j^+(t,x)$};

\draw[<-,blue,very thick] (5,-1.5) -- (8,-1.5);
\node[blue,below] at (6.5,-1.5) {$w_j^-(t,x)$};

\draw[blue,thick]
    (4.5,-0.75) arc (-180:-135:1.1);
\draw[blue,->,thick]
    (4.5,-0.75) arc (-180:-225:1.1);
\node[blue,right] at (4.5,-0.75) {$Q_{jj}$};

\draw[red,thick]
    (8.5,-0.75) arc (0:45:1.1);
\draw[red,->,thick]
    (8.5,-0.75) arc (0:-45:1.1);
\node[red,left] at (8.5,-0.75) {$R_{jj}$};

\draw[->,green!50!black!90,very thick]
    (3.5,0) -- (4.5,0);
\node[green!50!black!90,above] at (4,0) {$Q_{ji}$};

\draw[<-,green!50!black!90,very thick]
    (3.5,-1.5) -- (4.5,-1.5);
\node[green!50!black!90,below] at (4,-1.5) {$R_{ij}$};

\draw[<->,very thick] (-0.5,-3) -- (3.5,-3);
\node[above] at (0,-3) {$0$};
\node[above] at (3,-3) {$1$};

\draw[<->,very thick] (4.5,-3) -- (8.5,-3);
\node[above] at (5,-3) {$0$};
\node[above] at (8,-3) {$1$};

\draw[->,very thick] (4,-3.5) -- (4,-4.7);
\node[right] at (4,-4.1) {Fold};

\node[font=\large] at (4,-5.5) {After folding};

\draw[->,blue,very thick] (1,-6.5) -- (7,-6.5);
\node[blue,above] at (4,-6.5) {$w_i^-(t,x)$};

\draw[->,red,very thick] (1,-7.7) -- (7,-7.7);
\node[red,above] at (4,-7.7) {$w_j^+(t,x)$};

\draw[<-,red,very thick] (1,-9.3) -- (7,-9.3);
\node[red,below] at (4,-9.3) {$w_i^+(t,x)$};

\draw[<-,blue,very thick] (1,-10.5) -- (7,-10.5);
\node[blue,below] at (4,-10.5) {$w_j^-(t,x)$};


\draw[red,->,thick]
    (0.95,-9.3)
    .. controls (-0.6,-8.5) and (-0.6,-7.3) ..
    (0.95,-6.5);
\node[red,right] at (0.15,-7.4) {$R_{ii}$};

\draw[blue,->,thick]
    (7.05,-6.5)
    .. controls (8.6,-7.3) and (8.6,-8.5) ..
    (7.05,-9.3);
\node[blue,right] at (7.55,-7.9) {$Q_{ii}$};

\draw[blue,->,thick]
    (0.95,-10.5)
    .. controls (-0.9,-9.7) and (-0.9,-8.5) ..
    (0.95,-7.7);
\node[blue,left] at (0.55,-9.6) {$Q_{jj}$};

\draw[red,->,thick]
    (7.05,-7.7)
    .. controls (8.2,-8.4) and (8.2,-9.8) ..
    (7.05,-10.5);
\node[red,right] at (7.95,-9.1) {$R_{jj}$};

\draw[green!50!black!90,->,very thick]
    (0.95,-9.3)
    .. controls (-0.15,-8.8) and (-0.15,-8.2) ..
    (0.95,-7.7);
\node[green!50!black!90,right] at (0.15,-8.45) {$Q_{ji}$};

\draw[green!50!black!90,->,very thick]
    (0.95,-10.5)
    .. controls (-1.6,-9.4) and (-1.6,-7.6) ..
    (0.95,-6.5);
\node[green!50!black!90,right] at (-1.45,-8.5) {$R_{ij}$};

\draw[<->,very thick] (0.5,-11.8) -- (7.5,-11.8);
\node[above] at (1,-11.8) {$0$};
\node[above] at (7,-11.8) {$1$};

    \end{tikzpicture}%
    }

    \caption{Folding transformation of subsystem \(i\) onto subsystem \(j\).}
    \label{fig:folding_transformation}
\end{figure}
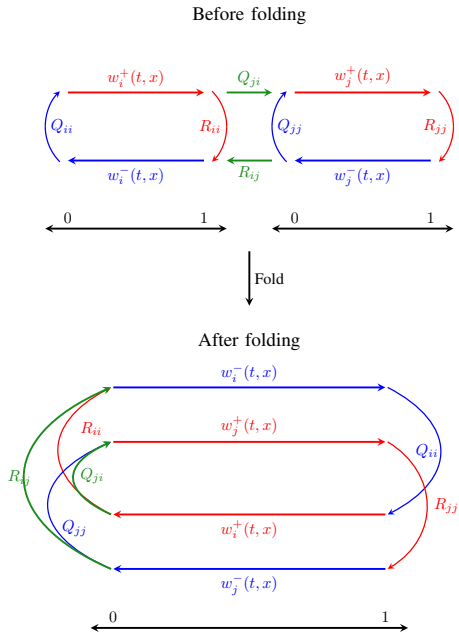

\subsection{Cutting Transform}
The folding methodology presented in Section~\ref{section:folding} applies only if the associated signed-graph is balanced. As explained in Lemma~\ref{lem:cut_odd}, using cutting transformations, it is always possible to obtain a port-network associated with a balanced signed graph.

\begin{lem}[Cutting transform]\label{lem:cut_odd}
Any port-network of the form~\eqref{eq:system_i_graph} satisfying
Assumption~\ref{assump:pairwise_consistency} can be transformed, by a finite
sequence of cutting operations, into a network whose associated signed graph
$\widetilde{\mathcal{G}}$ is balanced.
Each cutting operation is an {isometric isomorphism.} 
\end{lem}
The proof of Lemma~\ref{lem:cut_odd} is postponed to Appendix~\ref{appendix:proof_lem_cut}. The formal definition of a cutting transform is also given in the same appendix. We refer to Figure~\ref{fig:cutting_transformation} for a representation of a cutting transform.
\begin{figure}[H]
\centering
\scalebox{0.6}{
\begin{tikzpicture}

\draw (3.0,1.6) node{\large Before cutting};

\draw [>=stealth,->,red,very thick] (0,0) -- (6,0);
\draw [red] (4,0) node[above] {$w_i^+(t,x)$};

\draw [>=stealth,<-,blue,very thick] (0,-1.5) -- (6,-1.5);
\draw [blue] (2,-1.5) node[below] {$w_i^-(t,x)$};

\draw [blue,>=stealth,thick](-0.5,-0.75) arc (-180:-135:1.1);
\draw [blue,>=stealth,->,thick](-0.5,-0.75) arc (-180:-225:1.1);
\draw [blue] (-0.5,-0.75) node[right] {$Q_{ii}$};

\draw [red,>=stealth,thick](6.5,-0.75) arc (0:45:1.1);
\draw [red,>=stealth,->,thick](6.5,-0.75) arc (0:-45:1.1);
\draw [red] (6.5,-0.75) node[left] {$R_{ii}$};

\draw [dashed,very thick] (3,1) -- (3,-2.3);

\draw [>=stealth,<->,very thick] (-0.5,-3) -- (6.5,-3);
\draw (0,-3) node[above] {$0$};
\draw (3,-3) node[above] {$\frac{1}{2}$};
\draw (6,-3) node[above] {$1$};

\draw [>=stealth,->,very thick] (3,-3.8) -- (3,-5.0);
\draw (3.7,-4.4) node {Cut};

\draw (3.0,-6.1) node{\large After cutting};

\draw [>=stealth,->,red,very thick] (0,-7.6) -- (3,-7.6);
\draw [red] (1.5,-7.6) node[above] {$w_{i,1}^+(t,x)$};

\draw [>=stealth,<-,blue,very thick] (0,-9.1) -- (3,-9.1);
\draw [blue] (1.5,-9.1) node[below] {$w_{i,1}^-(t,x)$};

\draw [>=stealth,->,red,very thick] (5,-7.6) -- (8,-7.6);
\draw [red] (6.5,-7.6) node[above] {$w_{i,2}^+(t,x)$};

\draw [>=stealth,<-,blue,very thick] (5,-9.1) -- (8,-9.1);
\draw [blue] (6.5,-9.1) node[below] {$w_{i,2}^-(t,x)$};

\draw [blue,>=stealth,thick](-0.5,-8.35) arc (-180:-135:1.1);
\draw [blue,>=stealth,->,thick](-0.5,-8.35) arc (-180:-225:1.1);
\draw [blue] (-0.5,-8.35) node[right] {$Q_{ii}$};

\draw [red,>=stealth,thick](8.5,-8.35) arc (0:45:1.1);
\draw [red,>=stealth,->,thick](8.5,-8.35) arc (0:-45:1.1);
\draw [red] (8.5,-8.35) node[left] {$R_{ii}$};

\draw [>=stealth,->,green!50!black!90,very thick] (3.5,-7.6) -- (4.5,-7.6);
\draw [color=green!50!black!90] (4,-7.6) node[above] {$1$};

\draw [>=stealth,<-,green!50!black!90,very thick] (3.5,-9.1) -- (4.5,-9.1);
\draw [color=green!50!black!90] (4,-9.1) node[below] {$1$};

\draw [>=stealth,<->,very thick] (-0.5,-10.6) -- (3.5,-10.6);
\draw (0,-10.6) node[above] {$0$};
\draw (3,-10.6) node[above] {$1$};

\draw [>=stealth,<->,very thick] (4.5,-10.6) -- (8.5,-10.6);
\draw (5,-10.6) node[above] {$0$};
\draw (8,-10.6) node[above] {$1$};

\end{tikzpicture}}
\caption{Cutting transformation.}
\label{fig:cutting_transformation}
\end{figure}
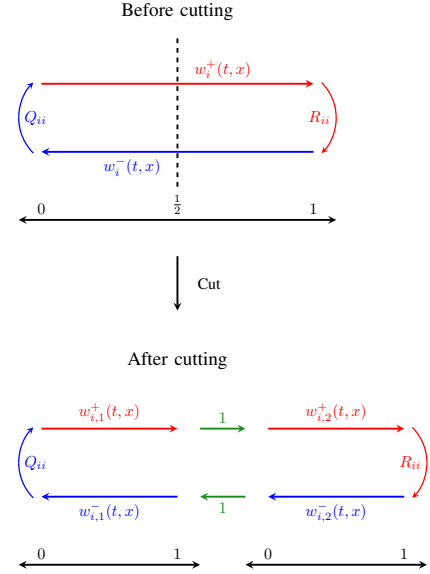
\section{Numerical Simulations for a Network}
\label{section:numerics}
{This section illustrates the applicability of the proposed method through numerical simulations on a network example. In this example, we consider a cycle represented by Figure~\ref{fig:port_network_considered} and the associated signed graph in Figure~\ref{fig:associated_sign_graph}. The network is therefore composed of four $1+1$ subsystems.

One control input acts on the first subsystem at $w_1^+(t,0)$, the other acts in the domain of the third subsystem. After applying a folding transform as explained in Section~\ref{section:folding}, one can obtain a system of the same form as~\eqref{eq:hyperbolic_couple} with
$n=m=4$, $d=2$, $B_0=0\in \R^{4\times 2}$, $B_1=e_1 e_1^\star$, where $(e_i)_{i=1}^4$ is the canonical basis of $\R^4$. The different couplings are chosen such that the different matrices of the folded subsystem satisfy
 \[
\Sigma^{++}=0, \quad \Sigma^{--}=0 \in \mathbb{R}^{6\times 6},
\]}
\[
\Sigma^{+-}
=
18\,\operatorname{diag}(-0.04,\,0.00,\,0.15,\,0.15),\]
\[
\Sigma^{-+}
=
18\,\operatorname{diag}(0.18,\,0.15,\,0,\,0.15),
\]
\[
\Lambda^+
=
\operatorname{diag}(0.5,\,1.40,\,2.3,\,2\pi),\]
\[
\Lambda^-
=
\operatorname{diag}(1,\,1.8,\,2.7,\,2\pi+0.4),
\]
\[
Q=
0.12\times\begin{pmatrix}
1&1&0&1\\
1&1&0&0\\
0&0&1&1\\
1&0&1&1
\end{pmatrix},
R=
0.08 \times\begin{pmatrix}
1&0&0&0\\
0&1&1&0\\
0&1&1&0\\
0&0&0&1
\end{pmatrix}.
\] 
The spatial distribution of the in-domain actuation is defined by \[
h(x)=
\begin{cases}
\sin(x)\,e_3 e_2^{\star},
& 0.4\le x\le 0.6,\\
\mathbf 0\in \R^{8 \times 2},
& \text{otherwise}.
\end{cases}
\]
Although the original graph has few vertices, the folded system~\eqref{eq:hyperbolic_couple} has dimension $n+m=8$, making the kernel and gain computations numerically demanding. These computations are performed offline and are largely parallelizable. The backstepping kernels are obtained by successive approximations based on a fixed-point argument~\cite{krstic2008boundary}, while the gains $g$ and $f$ defining $U_1$ in~\eqref{eq:U1_form} are computed by a least-squares procedure inspired by~\cite{braun_corona}. The PDE dynamics are simulated in Python using a finite-volume scheme on a uniform grid with $n_x=50$ points over $[0,1]$.
Figures~\ref{fig:cycle_dynamic} and~\ref{fig:cycle_control} show the open-loop and closed-loop dynamics, together with the corresponding control inputs. The simulations show that the open-loop network is unstable, whereas the proposed controller exponentially stabilizes both the system state and the control dynamics. The smaller control effort associated with $U_2$ reflects its auxiliary role in enforcing the stabilizability condition~\eqref{eq:max_rank_final} after reduction to the single input~$U_1$, which carries the main stabilizing action. Using $U_2$ as a stabilizing input may reduce the control effort further, but would require an extension of the interpolation/corona formulation used in the proof of Theorem~\ref{main_result} and~\cite{braun_corona} to the case $d>1$.
\begin{figure}[H]
    \centering
    \includegraphics[width=\columnwidth]{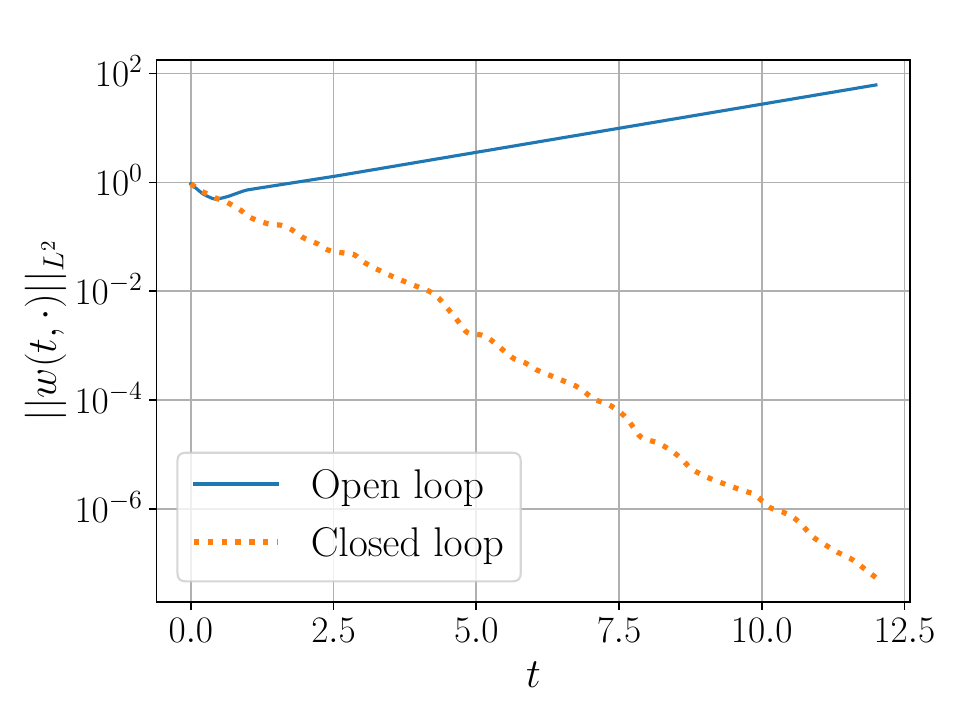}
    \caption{Dynamics of the cycle network.} 
    \label{fig:cycle_dynamic}
\end{figure}
\begin{figure}[H]
    \centering
    \includegraphics[width=\columnwidth]{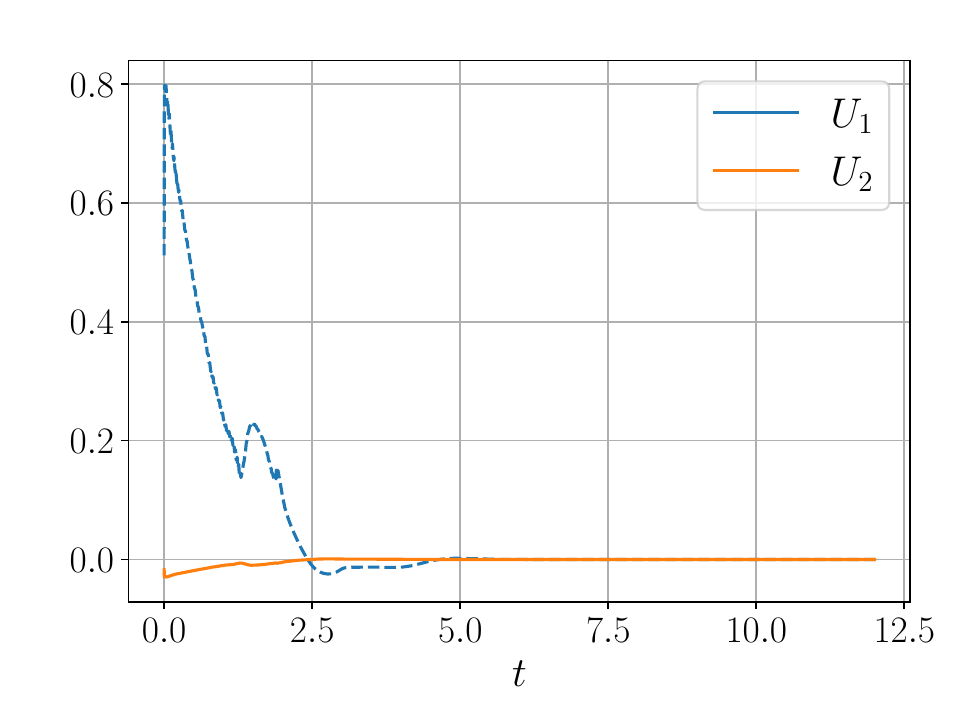}
    \caption{Control inputs} 
    \label{fig:cycle_control}
\end{figure}
\section*{Conclusion}
This paper introduced a new methodology for stabilizing a general class of first-order linear hyperbolic balance laws with partial actuation at the boundary or in the spatial domain. The framework was further applied to networks of such systems, including configurations involving cycles. Numerical simulations on a representative example illustrated the applicability of the proposed approach. Future research directions include the design of state observers and a deeper analysis of the connection between graph-theoretic properties of the network and the stabilizability condition introduced in this work. A natural conjecture is that the presence of a non-actuated cycle in the network prevents the stabilizability condition from being satisfied. A key direction for future research is the design of more efficient numerical methods enabling the treatment of larger networks and enhancing the scalability of the proposed framework.
\appendices
\makeatletter
\renewcommand{\@IEEEthmcounterin}[1]{\Alph{#1}}
\makeatother
\section{Proof of Lemma~\ref{lem:backsteppin_transform}}

\label{section:backstep_transform}
\begin{proof}
    The backstepping transform $\mathcal T$ will be the composition of two backstepping transforms previously introduced in~\cite{auriol2019explicit}.
    \paragraph{First transformation: removing in-domain couplings}
Following~\cite{coron2017finite,hu2016control}, introduce the Volterra backstepping
transformation
\begin{equation} \label{BS_transf}
\chi(t,x) = w(t,x) - \int_0^x K(x,y)\,w(t,y)\,dy, \qquad x\in[0,1],
\end{equation}
where $\chi=(\chi_+^\star,\chi_-^\star)^\star\in\mathbb{R}^{n+m}$ and
\[
K = \begin{pmatrix} K^{++} & K^{+-}\\ K^{-+} & K^{--} \end{pmatrix}
\]
is defined on $\mathcal{D}=\{(x,y)\in[0,1]^2:x\geq y\}$, with $K^{++}(x,y)\in\mathbb{R}^{n\times n}$,
$K^{+-}(x,y)\in\mathbb{R}^{n\times m}$, $K^{-+}(x,y)\in\mathbb{R}^{m\times n}$,
and $K^{--}(x,y)\in\mathbb{R}^{m\times m}$.
The kernel $K$ satisfies
\begin{align}
&\Lambda K_x(x,y) + K_y(x,y)\Lambda = -K(x,y)\Sigma, \label{eq:K_PDE}\\
&K(x,x)\Lambda - \Lambda K(x,x) = -\Sigma,\\&
\bigl(K^{--}(x,0)\Lambda^-\bigr)_{ij} = \bigl(K^{-+}(x,0)\Lambda^+Q\bigr)_{ij},~ i\geq j,
\label{eq:K_-}
\end{align}
with boundary conditions $K_{ij}^{++}(x,0)=0$ for $i\geq j$,
$K_{ij}^{++}(1,y)=0$ for $i<j$, and
$K_{ij}^{--}(1,y)=-\Sigma^{-}_{ij}/(\mu_i-\mu_j)$ for $i<j$.
This system admits a unique bounded piecewise continuously differentiable
solution~\cite{di2018stabilization,hu2016control}.
When spatially varying velocities are considered,
additional terms appear in~\eqref{eq:K_PDE} as in \cite{spatially_varying}; the well-posedness of the resulting
system is retained provided the velocities are continuously differentiable and
the entries of~$\Sigma$ are continuous.
Since~\eqref{BS_transf} is a Volterra transformation, it is boundedly
invertible~\cite{yoshida1960lectures}, with inverse
\begin{equation} \label{eq:inv_backstepping}
w(t,x) = \chi(t,x) + \int_0^x L(x,y)\,\chi(t,y)\,dy,
\end{equation}
where $L$ is piecewise continuously differentiable and satisfies the Volterra equation
\begin{equation} \label{eq:def_L}
L(x,y) = K(x,y) + \int_y^x K(x,\eta)\,L(\eta,y)\,d\eta,
\qquad (x,y)\in\mathcal{D},
\end{equation}
with block structure
\[
L(x,y) = \begin{pmatrix} L^{++}(x,y) & L^{+-}(x,y)\\ L^{-+}(x,y) & L^{--}(x,y) \end{pmatrix}.
\]
The transformation~\eqref{BS_transf} maps~\eqref{eq:hyperbolic_couple} into the
target system
\begin{equation} \label{eq:hyperbolic_target_1}
\left\{
\begin{aligned}
&\chi_t(t,x) + \Lambda\chi_x(t,x)
= \begin{pmatrix}G^{+}(x)\\G^{-}(x)\end{pmatrix}\chi_-(t,0) + h_\chi(x)\,U(t),\\[2mm]
&\chi_+(t,0) = Q\chi_-(t,0) + B_0 U(t),\\
&\chi_-(t,1) = R\chi_+(t,1) + B_1 U(t) + \int_0^1 W(\xi)\,\chi(t,\xi)\,d\xi,
\end{aligned}
\right.
\end{equation}
where, using the boundary term at $y=0$ arising from integration by parts
(with boundary condition $w^+(0)=Qw^-(0)+B_0U$),
\begin{align}
G^{+}(x) &:= K^{+-}(x,0)\Lambda^- - K^{++}(x,0)\Lambda^+Q, \label{eq:def_G_-+}\\
G^{-}(x) &:= K^{--}(x,0)\Lambda^- - K^{-+}(x,0)\Lambda^+Q, \label{eq:def_G_-}
\end{align}
$W(x)=\bigl(W^+(x)\ W^-(x)\bigr)$ with
\begin{align}
W^+(x) &:= -L^{-+}(1,x) + RL^{++}(1,x), \label{eq:def_N_+}\\
W^-(x) &:= -L^{--}(1,x) + RL^{-+}(1,x), \label{eq:def_N_-}
\end{align}
and $h_\chi(x)=\col(h_\chi^+(x),h_\chi^-(x))$ with
\begin{align*}
h_\chi^+(x) &:= h^+(x) - \int_0^x K^{++}(x,y)h^+(y)\,dy
\\
&- \int_0^x K^{+-}(x,y)h^-(y)\,dy - K^{++}(x,0)\Lambda^+B_0,
\\
h_\chi^-(x) &:= h^-(x) - \int_0^x K^{-+}(x,y)h^+(y)\,dy
\\
&- \int_0^x K^{--}(x,y)h^-(y)\,dy - K^{-+}(x,0)\Lambda^+B_0.
\end{align*}
By~\eqref{eq:K_-}, the matrix $G^{-}$ is strictly upper-triangular. The
transformation thus removes all in-domain couplings, shifting them to the
boundary (through $G^\pm$) and leaving a nonlocal residual term~$W$.

\paragraph{Second transformation: eliminating the nonlocal boundary term.}
We apply the integral transformation from~\cite{coron2017finite}, defined for
all $x\in[0,1]$ by
\begin{equation} \label{eq:backstepping_last}
\chi(t,x) = \gamma(t,x)
- \int_0^1\begin{pmatrix}0&0\\0&J(x,y)\end{pmatrix}\gamma(t,y)\,dy,
\end{equation}
where $J(x,y)\in\mathbb{R}^{m\times m}$ is a strictly upper-triangular matrix
whose entries are piecewise continuously differentiable on $[0,1]^2$,
satisfying~\cite{coron2017finite}
\begin{align}
&\Lambda^-J_x(x,y) + J_y(x,y)\Lambda^- = 0, \label{eq:F_PDE}\\
&J(x,0) = G^{-}(x)(\Lambda^-)^{-1},\quad J(0,y)=0. \label{eq:F_BC}
\end{align}
The inverse transformation reads
\begin{equation} \label{eq:backstepping_last_inv}
\gamma(t,x) = \chi(t,x)
- \int_0^1\begin{pmatrix}0&0\\0&\bar{J}(x,y)\end{pmatrix}\chi(t,y)\,dy,
\end{equation}
where $\bar{J}\in\mathbb{R}^{m\times m}$ is strictly upper-triangular with
piecewise continuously differentiable entries.
The transformation~\eqref{eq:backstepping_last} maps~\eqref{eq:hyperbolic_target_1}
into~\eqref{eq:hyperbolic_target_2}, where $H$ is the unique solution of the
Volterra equation
\begin{equation} \label{eq:def_G_bar}
H(x) = J(x,1)\Lambda^- + \int_0^1 J(x,y)\,H(y)\,dy.
\end{equation}
Moreover, $h_\gamma:=\col(h_\alpha,h_\beta)$ with
\begin{equation}
h_\alpha(x) := h_\chi^+(x), \label{eq:halpha}
\end{equation}
and $h_\beta$ is the unique solution of the Volterra equation
\begin{equation}
h_\beta(x) = h_\chi^-(x) + \int_0^1 J(x,y)\,h_\beta(y)\,dy, \label{eq:hbeta}
\end{equation}
and 
\begin{align}
F_\alpha(y) &= W^+(y), \label{eq:Falpha}\\
F_\beta(y) &= J(1,y) + W^-(y) - \int_0^1 W^-(\nu)\,J(\nu,y)\,d\nu. \label{eq:Fbeta}
\end{align}
Composing~\eqref{eq:backstepping_last_inv} and~\eqref{BS_transf}, the state
$\gamma$ is expressed in terms of $w$ by
\begin{equation} \label{eq:total_transformation}
\gamma(t,x) = (\mathcal{T}w(t))(x) := w(t,x) - \int_0^1 S(x,y)\,w(t,y)\,dy,
\end{equation}
where $S$ is a piecewise continuously differentiable kernel defined by
\begin{equation} \begin{aligned} \label{eq_def_H}
S(x,y) &= K(x,y)\,\mathbf{1}_{[0,x]}(y)
+ J_1(x,y)
\\
&- \int_0^1 J_1(x,\nu)\,K(\nu,y)\,\mathbf{1}_{[0,\nu]}(y)\,d\nu,
\end{aligned}
\end{equation}
with $J_1(x,y)=\begin{pmatrix}0&0\\0&\bar{J}(x,y)\end{pmatrix}$.
\end{proof}
\section{IDE related results}
\subsection{Proof of Lemma~\ref{lem_derivation_IDE_closed_loop}} \label{Appendix_IDE}
\begin{proof}
        That $U$ satisfies equations~\eqref{eq:Ui_rank_reduction}-\eqref{eq:U1_form} follows directly from the application of the method of characteristics to the pure transport system~\eqref{eq:dynamical_feedback}. The derivation for $X$ proceeds in a similar manner and is detailed below for completeness. The computations are similar to those performed in~\cite{auriol2019explicit}.
Since we are not interested in the transient behavior of~\eqref{eq:hyperbolic_target_2}, we assume zero initial conditions throughout. Applying the Laplace transform to~\eqref{eq:hyperbolic_target_2} then yields, for all $s \in \C$,
\begin{equation}
\label{eq:laplace-input}
\left\{
\begin{aligned}
&s\widehat\alpha(s,x)+\Lambda^+ \partial_x\widehat\alpha(s,x)
=G(x)\widehat\beta(s,0)+h_\alpha(x)\widehat U(s),\\
&s\widehat\beta(s,x)-\Lambda^-\partial_x\widehat\beta(s,x)
=H(x)\widehat X(s)+h_\beta(x)\widehat U(s),\\
&\widehat\alpha(s,0)=Q\widehat\beta(s,0)+B_0\widehat U(s),\\
&\widehat X(s)=R\widehat\alpha(s,1)\\
&+\int_0^1\Big(F_\alpha(\nu)\widehat\alpha(s,\nu)+F_\beta(\nu)\widehat\beta(s,\nu)\Big)\,d\nu+B_1\widehat U(s).
\end{aligned}
\right.
\end{equation}
\paragraph{Step 1: Solve the equation for \(\beta\).}
For each \(j\in\{1,\dots,m\}\),
\begin{align*}
&s\widehat\beta_j(s,x)-\mu_j\partial_x\widehat\beta_j(s,x)
=
\big(H(x)\widehat X(s)\big)_j+\big(h_\beta(x)\widehat U(s)\big)_j,\\&
\widehat\beta_j(s,1)=\widehat X_j(s).
\end{align*}
Solving yields
\begin{align}
\label{eq:beta-laplace-input-component}
\widehat\beta_j(s,x)
&=
e^{-s(1-x)/\mu_j}\widehat X_j(s)
\\
&+\int_x^1 \frac1{\mu_j}e^{-s(\xi-x)/\mu_j}\big(H(\xi)\widehat X(s)\big)_j\,d\xi
\\
&+\int_x^1 \frac1{\mu_j}e^{-s(\xi-x)/\mu_j}(h_\beta(\xi)\widehat U(s))_j\,d\xi\;. \nonumber
\end{align}
Let us consider the selector matrices
\[
P_j:=l_j l_j^\star\in\R^{m\times m},\qquad j=1,\dots,m,
\]
where $(l_j)$ are the vectors of the canonical basis of $\R^m$.
We have
$$P_j y = y_j l_j, \qquad y\in \R^m.$$
Hence,
\begin{align}
\label{eq:beta-laplace-input}
&\widehat\beta(s,x)
=\sum_{j=1}^m\int_x^1 \frac1{\mu_j}e^{-s(\xi-x)/\mu_j}P_jh_\beta(\xi)\,d\xi \widehat U(s) 
\\&+\big(
\sum_{j=1}^m e^{-s(1-x)/\mu_j}P_j
\\&+
\sum_{j=1}^m\int_x^1 \frac1{\mu_j}e^{-s(\xi-x)/\mu_j}P_jH(\xi)\,d\xi
\big)\widehat X(s). \nonumber
\end{align}
Taking inverse Laplace transforms, we obtain 
\begin{equation}
\begin{aligned}
\label{eq:beta-x-time-input}
\beta(t,x)
&=
\sum_{j=1}^m P_jX\!\left(t-\frac{1-x}{\mu_j}\right)
+\int_0^{\frac1{\mu_1}}N_{\beta,x}(\nu)X(t-\nu)\,d\nu
\\
&+\int_0^{\frac1{\mu_1}}M_{\beta,x}(\nu)U(t-\nu)\,d\nu,
\end{aligned}
\end{equation}
with
\begin{equation}
\label{eq:Nbeta-x-input}
N_{\beta,x}(\nu)
=
\sum_{j=1}^m P_jH(x+\mu_j\nu)\,\mathbf 1_{[0,\,(1-x)/\mu_j]}(\nu),
\end{equation}
and
\begin{equation}
\label{eq:Mbeta-x}
M_{\beta,x}(\nu)
=
\sum_{j=1}^m P_jh_\beta(x+\mu_j\nu)\,\mathbf 1_{[0,\,(1-x)/\mu_j]}(\nu).
\end{equation}
Indeed, for each \(j\), with \(\xi=x+\mu_j\nu\),
$
\int_x^1 \frac1{\mu_j}e^{-s(\xi-x)/\mu_j}P_jh_\beta(\xi)\,d\xi
=
\int_0^{(1-x)/\mu_j} e^{-s\nu}P_jh_\beta(x+\mu_j\nu)\,d\nu,
$

$
    \int_x^1 \frac{1}{\mu_j} e^{-s(\xi-x)/\mu_j} P_j H(\xi)\,d\xi
=
\int_0^{(1-x)/\mu_j} e^{-s\nu} P_j H(x+\mu_j\nu )\,d\nu.
$
\paragraph{Step 2: Solve the equation for $\alpha$}

For each \(i\in\{1,\dots,n\}\), we have
\[
s\widehat\alpha_i(s,x)+\lambda_i\partial_x\widehat\alpha_i(s,x)
=
\big(G(x)\widehat\beta(s,0)\big)_i+\big(h_\alpha(x)\widehat U(s)\big)_i.
\]
with
\[
\widehat\alpha(s,0)=Q\widehat\beta(s,0)+B_0\widehat U(s).
\]
Therefore,
\begin{equation}
\begin{aligned}
\label{eq:alpha-laplace-input-component}
&\widehat\alpha_i(s,x)
=
e^{-sx/\lambda_i}\big(Q\widehat\beta(s,0)+B_0\widehat U(s)\big)_i
\\&+
\int_0^x \frac1{\lambda_i}e^{-s(x-\xi)/\lambda_i}
\big(G(\xi)\widehat\beta(s,0)+h_\alpha(\xi)\widehat U(s)\big)_i\,d\xi.
\end{aligned}
\end{equation}
As before, let \((e_i)\) be the canonical basis of \(\mathbb R^n\) and define
\[
E_i:=e_ie_i^\star\in\mathbb R^{n\times n},\qquad i=1,\dots,n.
\]
We obtain in the time domain
\begin{equation}
\label{eq:alpha-time-raw-input}
\begin{aligned}
\alpha(t,x)
={}&
\sum_{i=1}^n E_iQ\,\beta\!\left(t-\frac{x}{\lambda_i},0\right)
+
\sum_{i=1}^n E_iB_0\,U\!\left(t-\frac{x}{\lambda_i}\right)
\\
&+
\sum_{i=1}^n\int_0^x \frac1{\lambda_i}E_iG(\xi)\,
\beta\!\left(t-\frac{x-\xi}{\lambda_i},0\right)\,d\xi
\\&+
\sum_{i=1}^n\int_0^x \frac1{\lambda_i}E_i h_\alpha(\xi)\,
U\!\left(t-\frac{x-\xi}{\lambda_i}\right)\,d\xi.
\end{aligned}
\end{equation}

Substituting \eqref{eq:beta-x-time-input} (at $x=0$) into \eqref{eq:alpha-time-raw-input} yields a decomposition into point delays and distributed terms, both in \(X\) and in \(U\).
\begin{equation}
\label{eq:alpha-final-input}
\begin{aligned}
&\alpha(t,x)
=
\sum_{i=1}^n\sum_{j=1}^m
E_iQP_j X\!\left(t-\frac{x}{\lambda_i}-\frac1{\mu_j}\right)
\\
&+\int_0^{x/\lambda_{1}+\frac{1}{\mu_{1}}}N_{\alpha,x}(\nu)X(t-\nu)\,d\nu
+\sum_{i=1}^n E_iB_0\,U\!\left(t-\frac{x}{\lambda_i}\right)
\\&+\int_0^{x/\lambda_{1}+\frac{1}{\mu_{1}}}M_{\alpha,x}(\nu)U(t-\nu)\,d\nu.
\end{aligned}
\end{equation}
with
\begin{equation}
\label{eq:Nalpha-x}
\begin{aligned}
&N_{\alpha,x}(\nu)
=
\sum_{i=1}^n
E_i Q\,N_{\beta,0}\!\left(\nu-\frac{x}{\lambda_i}\right)
\ind_{\left[\frac{x}{\lambda_i},\,\frac{x}{\lambda_i}+\frac{1}{\mu_{1}}\right]}(\nu) \quad 
\\
&+
\sum_{i=1}^n\sum_{j=1}^m
E_i G\!\big(x-\lambda_i(\nu-\tfrac{1}{\mu_j})\big)P_j\,
\ind_{\left[\frac{1}{\mu_j},\,\frac{1}{\mu_j}+\frac{x}{\lambda_i}\right]}(\nu)
\\ 
&+
\sum_{i=1}^n
\int_0^x \frac{1}{\lambda_i} E_i G(\xi)\,
N_{\beta,0}\!\left(\nu-\frac{x-\xi}{\lambda_i}\right)
\ind_{\left[\frac{x-\xi}{\lambda_i},\,\frac{x-\xi}{\lambda_i}+\frac{1}{\mu_{1}}\right]}(\nu)\,d\xi 
\end{aligned}
\end{equation}
and,
\begin{equation}
\label{eq:Malpha-x}
\begin{aligned}
&M_{\alpha,x}(\nu)
=
\sum_{i=1}^n
E_iQ\,M_{\beta,0}\!\left(\nu-\frac{x}{\lambda_i}\right)
\mathbf 1_{\left[\frac{x}{\lambda_i},\,\frac{x}{\lambda_i}+\frac{1}{\mu_1}\right]}(\nu)
\\
&+
\sum_{i=1}^n
E_i h_\alpha(x-\lambda_i\nu)\mathbf 1_{[0,\,x/\lambda_i]}(\nu)
\\
&+
\sum_{i=1}^n
\int_0^x \frac1{\lambda_i}E_iG(\xi)\,
M_{\beta,0}\!\left(\nu-\frac{x-\xi}{\lambda_i}\right)
\mathbf 1_{\left[\frac{x-\xi}{\lambda_i},\,\frac{x-\xi}{\lambda_i}+\frac{1}{\mu_1}\right]}(\nu)\,d\xi.
\end{aligned}
\end{equation}
\paragraph{Step 3: Pointwise part in the IDE}
We now set \(x=1\) in \eqref{eq:alpha-final-input}.
\begin{equation}
\label{eq:Ralpha1-input-expanded}
\begin{aligned}
&R\alpha(t,1)
=
\sum_{i=1}^n\sum_{j=1}^m
RE_iQP_j X\!\left(t-\frac1{\lambda_i}-\frac1{\mu_j}\right)
\\&+\int_0^{\tau_*}RN_{\alpha,1}(\nu)X(t-\nu)\,d\nu
\\
&+
\sum_{i=1}^n RE_iB_0\,U\!\left(t-\frac1{\lambda_i}\right)
+\int_0^{\tau_*}RM_{\alpha,1}(\nu)U(t-\nu)\,d\nu,
\end{aligned}
\end{equation}
where
\[
\tau_* := \max_{\substack{1\le i\le n\\1\le j\le m}}\tau_{ij} = \frac{1}{\lambda_1}+\frac{1}{\mu_1} =: \tau_{11},
\]
with 
\begin{equation}
\label{eq:tauij-input}
\tau_{ij}:=\frac1{\lambda_i}+\frac1{\mu_j},
\qquad i=1,\dots,n,\ \ j=1,\dots,m.
\end{equation}
Let \(\{\tau_k\}_{k=1}^l\) be the set of distinct values among the \(\tau_{ij}\), and define
\begin{equation}
\label{eq:Ak-input}
A_k:=\sum_{(i,j):\,\tau_{ij}=\tau_k}RE_iQP_j.
\end{equation}
Then
\begin{equation}
\label{eq:point-delay-sum-input}
\sum_{i=1}^n\sum_{j=1}^m
RE_iQP_j X(t-\tau_{ij})
=
\sum_{k=1}^l A_kX(t-\tau_k).
\end{equation}
We also define the input point-delay matrices
\begin{equation}
\label{eq:Btildei}
\widetilde B_i:=RE_iB_0,\qquad i=1,\dots,n.
\end{equation}
Thus
\begin{equation}
\label{eq:Ralpha-total-input}
\begin{aligned}
&R\alpha(t,1)
=
\sum_{k=1}^l A_kX(t-\tau_k)
+\int_0^{\tau_*}N_R(\nu)X(t-\nu)\,d\nu
\\&+\sum_{i=1}^n \widetilde B_i U\!\left(t-\frac1{\lambda_i}\right)
+\int_0^{\tau_*}M_R(\nu)U(t-\nu)\,d\nu,
\end{aligned}
\end{equation}
where
\begin{equation}
\label{eq:NR-MR}
N_R(\nu):=RN_{\alpha,1}(\nu),
\qquad
M_R(\nu):=RM_{\alpha,1}(\nu).
\end{equation}
Identifying $\theta_i := 1/\lambda_i$ for $i=1,\dots,n$,  equation~\eqref{eq:Ralpha-total-input} yields the desired form.
\paragraph{Step 4: Distributed part for the IDE.}
We have
$
X(t)=R\alpha(t,1)+\int_0^1F_\alpha(x)\alpha(t,x)\,dx+\int_0^1F_\beta(x)\beta(t,x)\,dx+B_1U(t).
$

\paragraph{Contribution of \(\int_0^1F_\beta(x)\beta(t,x)\,dx\).}
Using \eqref{eq:beta-x-time-input}, the pointwise state part gives
\begin{equation}
\label{eq:NFbeta-pt-input}
N_{F_\beta}^{pt}(\nu)
=
\sum_{j=1}^m
\mu_jF_\beta(1-\mu_j\nu)P_j\,\mathbf 1_{[0,\,1/\mu_j]}(\nu),
\end{equation}
since with $x = 1-\mu_j\nu,$

$
\int_0^1F_\beta(x)P_jX\!\left(t-\frac{1-x}{\mu_j}\right)\,dx
=
\int_0^{1/\mu_j}\mu_jF_\beta(1-\mu_j\nu)P_jX(t-\nu)\,d\nu.
$
The distributed state part gives
\begin{equation}
\label{eq:NFbeta-db-input}
N_{F_\beta}^{db}(\nu)
=
\int_0^1F_\beta(x)N_{\beta,x}(\nu)\,dx.
\end{equation}
The input contribution gives
\begin{equation}
\label{eq:MFbeta}
M_{F_\beta}(\nu)
=
\int_0^1F_\beta(x)M_{\beta,x}(\nu)\,dx.
\end{equation}
Therefore,
\begin{equation}
\label{eq:Fbeta-total-input}
\begin{aligned}
\int_0^1F_\beta(x)\beta(t,x)\,dx
&=
\int_0^{\frac{1}{\mu_1}}\Big(N_{F_\beta}^{pt}(\nu)+N_{F_\beta}^{db}(\nu)\Big)X(t-\nu)\,d\nu
\\&+
\int_0^{\frac{1}{\mu_1}}M_{F_\beta}(\nu)U(t-\nu)\,d\nu.
\end{aligned}
\end{equation}

\paragraph{Contribution of \(\int_0^1F_\alpha(x)\alpha(t,x)\,dx\).}
Using \eqref{eq:alpha-final-input}, the pointwise state part gives
\begin{equation}
\label{eq:NFalpha-pt-input}
N_{F_\alpha}^{pt}(\nu)
=
\sum_{i=1}^n\sum_{j=1}^m
\lambda_iF_\alpha\!\big(\lambda_i(\nu-\tfrac1{\mu_j})\big)
E_iQP_j\,
\mathbf 1_{\left[\frac1{\mu_j},\,\frac1{\mu_j}+\frac1{\lambda_i}\right]}(\nu),
\end{equation}
because
\begin{align*}
&\int_0^1F_\alpha(x)E_iQP_jX\!\left(t-\frac{x}{\lambda_i}-\frac1{\mu_j}\right)\,dx
\\&=
\int_{1/\mu_j}^{1/\mu_j+1/\lambda_i}
\lambda_iF_\alpha\!\big(\lambda_i(\nu-\tfrac1{\mu_j})\big)
E_iQP_jX(t-\nu)\,d\nu.
\end{align*}
The distributed state part is
\begin{equation}
\label{eq:NFalpha-db-input}
N_{F_\alpha}^{db}(\nu)
=
\int_0^1F_\alpha(x)N_{\alpha,x}(\nu)\,dx.
\end{equation}
The pointwise input part gives
\begin{equation}
\label{eq:MFalpha-ptU}
M_{F_\alpha}^{pt,U}(\nu)
=
\sum_{i=1}^n
\lambda_iF_\alpha(\lambda_i\nu)E_iB_0\,\mathbf 1_{[0,\,1/\lambda_i]}(\nu),
\end{equation}
since with $x= \lambda_i \nu,$
\begin{align*}
&\int_0^1F_\alpha(x)E_iB_0\,U\!\left(t-\frac{x}{\lambda_i}\right)\,dx
\\&=
\int_0^{1/\lambda_i}\lambda_iF_\alpha(\lambda_i\nu)E_iB_0\,U(t-\nu)\,d\nu.
\end{align*}
The distributed input part is
\begin{equation}
\label{eq:MFalpha-dbU}
M_{F_\alpha}^{db,U}(\nu)
=
\int_0^1F_\alpha(x)M_{\alpha,x}(\nu)\,dx.
\end{equation}
Combining these contributions and extending the kernels by zero on the possibly larger domain of definition $(0, \tau_*)$ when necessary to obtain a common support,
\begin{equation}
\begin{aligned}
\label{eq:Falpha-total-input}
&\int_0^1F_\alpha(x)\alpha(t,x)\,dx
=
\int_0^{\tau_*}\Big(N_{F_\alpha}^{pt}(\nu)+N_{F_\alpha}^{db}(\nu)\Big)X(t-\nu)\,d\nu
\\&+
\int_0^{\tau_*}\Big(M_{F_\alpha}^{pt,U}(\nu)+M_{F_\alpha}^{db,U}(\nu)\Big)U(t-\nu)\,d\nu.
\end{aligned}
\end{equation}
\paragraph{Step 5: Conclusion}
Combining \eqref{eq:Ralpha-total-input}, \eqref{eq:Fbeta-total-input},
and \eqref{eq:Falpha-total-input} yields~\eqref{eq:ide-input-target}, 
\begin{equation*}
\begin{aligned}
X(t)
={}&
\sum_{k=1}^l A_kX(t-\tau_k)
+\int_0^{\tau_*}N(\nu)X(t-\nu)\,d\nu
\\
&+
\sum_{i=1}^n \widetilde B_iU\!\left(t-\theta_i\right)
+\int_0^{\tau_*}M(\nu)U(t-\nu)\,d\nu
+B_1U(t),
\end{aligned}
\end{equation*}
with
\begin{equation}
\label{eq:N-final-input}
N(\nu)
=
N_R(\nu)
+
N_{F_\alpha}^{pt}(\nu)
+
N_{F_\alpha}^{db}(\nu)
+
N_{F_\beta}^{pt}(\nu)
+
N_{F_\beta}^{db}(\nu),
\end{equation}
and
\begin{equation}
\label{eq:MU-final}
M(\nu)
=
M_R(\nu)
+
M_{F_\alpha}^{pt,U}(\nu)
+
M_{F_\alpha}^{db,U}(\nu)
+
M_{F_\beta}(\nu).
\end{equation}
    \end{proof}
    \subsection{Proof of Proposition~\ref{prop:full_rank_ide}}
\label{section:preuve_prop_full_rank}
\begin{proof}
We argue by contraposition. Assume that \eqref{eq:ide-rank-ass} fails. Then there exists \(s\in\mathbb C\) with $\Re s>-\omega$ such that
\[
\operatorname{rank}[\qhat(s),-\phat(s)]<m.
\]
Hence there exists \(
\eta\in\mathbb C^m\setminus\{0\}
\)
such that
\begin{equation}
\label{eq:left-kernel}
\qhat(s)^\star\eta=0,
\qquad
\phat(s)^\star\eta=0.
\end{equation}
We construct $(\phi,\psi)\in D(\mathcal A_1^\star)\setminus\{(0,0)\}$ satisfying
\[
(s-\mathcal A_1^\star)(\phi,\psi)=0,
\qquad
\mathcal B_1^\star(\phi,\psi)=0.
\]
Using Proposition~\ref{prop:stabilizability_target2}, it will contradict~\eqref{eq:pde-stab-ass}.

\medskip

\noindent
\textbf{Step 1: construction of \(\phi\) and \(\psi\).}
Define \(\phi\) as the solution of
\begin{equation}
\label{eq:phi-adj-eq}
s\phi(x)=\Lambda^+\phi_x(x)+F_\alpha(x)^\star\eta,
\qquad
\Lambda^+\phi(1)=R^\star\eta.
\end{equation}
Since \(\Lambda^+\) is diagonal, this problem is solved componentwise and yields
\begin{equation}
\label{eq:phi-explicit}
\Lambda^+\phi(x)
=
e^{-s(1-x){\Lambda^+}^{-1}}R^\star\eta
+
\int_x^1 e^{-s(\xi-x)\Lambda^{+^{-1}}}F_\alpha(\xi)^\star\eta\,d\xi.
\end{equation}
In particular, \(\phi\in H^1((0,1);\mathbb C^n)\). Next define
\begin{equation}
\label{eq:Y0-def}
Y_0:=Q^\star\Lambda^+\phi(0)+\int_0^1 G(x)^\star\phi(x)\,dx,
\end{equation}
and let \(\psi\) be the solution of
\begin{equation}
\label{eq:psi-adj-eq}
s\psi(x)=-\Lambda^-\psi_x(x)+F_\beta(x)^\star\eta,
\qquad
\Lambda^-\psi(0)=Y_0.
\end{equation}
Again, since \(\Lambda^-\) is diagonal, this is solved componentwise:
\begin{equation}
\label{eq:psi-explicit}
\Lambda^-\psi(x)
=
e^{-sx{\Lambda^-}^{-1}}Y_0
+
\int_0^x e^{-s(x-\xi){\Lambda^-}^{-1}}F_\beta(\xi)^\star\eta\,d\xi.
\end{equation}
Hence \(\psi\in H^1((0,1);\mathbb C^m)\), and by construction
\begin{equation}
\label{eq:left-adj-bc}
Q^\star\Lambda^+\phi(0)-\Lambda^-\psi(0)+\int_0^1 G(x)^\star\phi(x)\,dx=0.
\end{equation}

\medskip

\noindent
\textbf{Step 2: identification of \(\zeta\).}
Set
\begin{equation}
\label{eq:zeta-def-proof}
\zeta(\phi, \psi):=\Lambda^-\psi(1)+\int_0^1 H(x)^\star\psi(x)\,dx.
\end{equation}
Repeating the same computations as in the derivation of the IDE, but now for the adjoint system, one obtains the transposed delay relation
\begin{equation}
\label{eq:zeta-transposed-clean}
\zeta(\phi, \psi)
=
\sum_{k=1}^l A_k^\star e^{-s\tau_k}\eta
+
\int_0^{\tau_*} N(\nu)^\star e^{-s\nu}\eta\,d\nu.
\end{equation}
Comparing with~\eqref{eq:Qs-def-proof}
we get
\begin{equation}
\label{eq:Q-transpose-zeta}
\qhat(s)^\star\eta=\eta-\zeta(\phi, \psi).
\end{equation}
Since \(\qhat(s)^\star\eta=0\) by \eqref{eq:left-kernel}, it follows that
\begin{equation}
\label{eq:zeta-equals-eta-goal}
\zeta(\phi, \psi)=\eta.
\end{equation}
Therefore \(\phi\) and \(\psi\) satisfy
\begin{equation}
\label{eq:adjoint-eig-final}
(s-A^\star)(\phi,\psi)=0.
\end{equation}
Moreover, using the same characteristic computations as before, we obtain
\begin{equation}
\label{eq:Cstar-bookkeeping}
\begin{aligned}
&B_0^\star \Lambda^+\phi(0)
+
\int_0^1\Big(h_\alpha(x)^\star\phi(x)+h_\beta(x)^\star\psi(x)\Big)\,dx
\\&=
\sum_{i=1}^n \widetilde B_i^\star e^{-s/\lambda_i}\eta
+
\int_0^{\tau_*} M(\nu)^\star e^{-s\nu}\eta\,d\nu.
\end{aligned}
\end{equation}
Since \(\zeta(\phi, \psi)=\eta\), we deduce from \eqref{eq:B1star-def} that
\begin{align}
\mathcal B_1^\star(\phi,\psi)
&=
B_1^\star\eta
+
\sum_{i=1}^n \widetilde B_i^\star e^{-s/\lambda_i}\eta
+
\int_0^{\tau_*} M(\nu)^\star e^{-s\nu}\eta\,d\nu
\notag\\
&=
\left(
B_1+\sum_{i=1}^n \widetilde B_i e^{-s/\lambda_i}
+\int_0^{\tau_*} M(\nu)e^{-s\nu}\,d\nu
\right)^\star\eta
\notag\\
&=
\phat(s)^\star\eta.
\label{eq:phat-transpose-control-detailed}
\end{align}
Hence, by \eqref{eq:left-kernel},
\[
\mathcal B_1^\star(\phi,\psi)=0.
\]

\medskip

\noindent
\textbf{Step 3: contradiction.}
We have constructed \((\phi,\psi)\in D(\mathcal A_1^\star)\) such that
\[
(s-\mathcal A_1^\star)(\phi,\psi)=0,
\qquad
\mathcal B_1^\star(\phi,\psi)=0.
\]
Moreover \((\phi,\psi)\neq(0,0)\). Indeed, if \((\phi,\psi)=(0,0)\), then \eqref{eq:zeta-def-proof} gives \(\zeta(\phi, \psi)=0\), hence \(\eta=0\) by \eqref{eq:zeta-equals-eta-goal}, a contradiction. Thus
\[
\ker(s-\mathcal A_1^\star)\cap\ker \mathcal B_1^\star\neq\{0\},
\]
which contradicts \eqref{eq:pde-stab-ass} using Proposition~\ref{prop:stabilizability_target2}. Therefore \eqref{eq:ide-rank-ass} must hold.
\end{proof}
\subsection{Proof of Theorem~\ref{thm:rank_reduction}}
\label{section:proof_thm_stable_rank}
We begin with some intermediate results. 
\begin{lem}\label{lem:dev_det_q}
There exists a compactly supported function $\Ntarget\in L^2(0, \infty)$ such that for all $s\in\mathbb{C}_{\omega_2}$, 

\[\det\widehat{q}(s) = \det\Delta_0(s) + \Lap \Ntarget(s),\]
with $\widehat{q}$ defined in~\eqref{eq:Qs-def-proof} and $\omega_2$ in~\eqref{eq:defomega2}.
\end{lem}
\begin{proof}
Writing $\widehat{q}$ column-wise as
\[
\widehat{q}(s) = \bigl[\Delta_{0,1}(s)-\widehat{N}_1(s)\ \cdots\ \Delta_{0,m}(s)-\widehat{N}_m(s)\bigr],
\]
where $\widehat{N}_i(s)$ is the Laplace transform of the $i$-th column of
$N\in\mathbb{R}^{m\times m}$ and $\Delta_{0,i}(s)$ is the $i$-th column of
$\Delta_0$ defined in~\eqref{def:pp}, multilinearity of the determinant gives
\[
\det\widehat{q}(s) = \det\Delta_0(s) + \Lap \Ntarget(s),
\]
with
\begin{equation}\label{def:Ntarget}
\Lap \Ntarget(s) := (-1)^m\det\widehat{N}(s)
+ \sum_{\substack{L\subset\{1,\dots,m\}\\ L\neq\varnothing,\,L\neq\{1,\dots,m\}}}
\det M_L(s),
\end{equation}
where $M_L(s)$ is the $m\times m$ matrix with columns $m_j$ defined for all $j \in \{1,\dots,m\}$
\[
m_j = \begin{cases}
\Delta_{0,j}(s) & \text{if } j\notin L,\\
-\widehat{N}_j(s) & \text{if } j\in L.
\end{cases}
\]
It remains to show that $\Lap \Ntarget$ is the Laplace transform of a function $N_0 \in L^2_{\omega_2}(0, \infty)$.
Each entry of $\widehat{N}$ is the Laplace transform of a compactly supported function in $L^2((0,\infty),\mathbb{R})$. Consequently, the inverse Laplace transform of $\det\widehat{N}$ is given in the time domain by a finite sum of convolutions of compactly supported $L^2$-functions; since convolution preserves both compact support and $L^2$-regularity, $\det\widehat{N}$ is the Laplace transform of an integrable function compactly supported in $(0, \infty)$. The same argument applies to each $\det M_L(s)$: its entries are products of entries of $\widehat{N}$ with entries of $\Delta_0$, which in the time domain correspond to convolutions of compactly supported $L^2$-functions with the measures associated with the inverse Laplace transform of $\Delta_0$. Hence each $\det M_L$ is the Laplace transform of a causal compactly supported function in $L^2((0,\infty),\mathbb{R})$. 
We conclude that $\Lap N_0$ is the Laplace transform of a compactly supported function $N_0 \in L^2(0, \infty)$.
\end{proof}

\begin{prop}\label{prop:Z0_fini}
Under Assumption~\ref{Assum_OL_stab}, the zero set of $\det\widehat{q}$ in
$\mathbb{C}_{\omega_2}$, namely
\[
Z_q := \bigl\{s\in\mathbb{C}_{\omega_2} : \det\widehat{q}(s)=0\bigr\},
\]
is finite.
\end{prop}
\begin{proof}
By Lemma~\ref{lem:dev_det_q}, $\det\widehat{q}(s) = \det\Delta_0(s) + \Lap \Ntarget(s)$ with a compactly supported function $\Lap N_0 \in L^2(0, \infty)$. By the Riemann-Lebesgue lemma~\cite[p.~103]{Rudin1991} or dominated convergence, $|\Lap \Ntarget(s)|\to 0$ as $|s|\to\infty$ in $\mathbb{C}_{\omega_2}$, so there exists a compact set $K\subset\mathbb{C}_{\omega_2}$ such that $|\Ntarget(s)|<\infpp$ for all $s\in\mathbb{C}_{\omega_2}\setminus K$, where $\infpp$ is from Corollary~\ref{coro:pp_ide_stable}. By~\eqref{inf:pp_positif}, $\det\widehat{q}(s)\neq 0$ outside $K$, so $Z_q\subset K$. Since $\det\widehat{q}$ is holomorphic and not identically zero (by Assumption~\ref{Assum_OL_stab}, its principal part $\det\Delta_0$ is nonzero), its zeros are isolated. Therefore $Z_q$, being a closed discrete subset of the compact set $K$, is finite.
\end{proof}

\begin{prop}\label{prop:Q_inv_bounded}
Under Assumption~\ref{Assum_OL_stab}, there exists a compact
$K_{\omega_2}\subset\mathbb{C}_{\omega_2}$ such that $s\mapsto\widehat{q}^{-1}(s)$
exists, is holomorphic, and is bounded on $\mathbb{C}_{\omega_2}\setminus K_{\omega_2}$.
\end{prop}
\begin{proof}
By Proposition~\ref{prop:Z0_fini}, $Z_q$ is finite, so there exists a compact
$K_{\omega_2}\subset\mathbb{C}_{\omega_2}$ with $Z_q\subset K_{\omega_2}$.
Up to enlarging $K_{\omega_2}$, \cite[Lemma~5]{braun2026spectralexponentialstabilitycriterion}
gives
\[
\inf_{s\in\mathbb{C}_{\omega_2}\setminus K_{\omega_2}}|\det\widehat{q}(s)|>0.
\]
Hence there exists $C>0$ such that for all $s\in\mathbb{C}_{\omega_2}\setminus K_{\omega_2}$,
\[
\|\widehat{q}^{-1}(s)\|
= \left\|\frac{\operatorname{adjugate}(\widehat{q}(s))}{\det\widehat{q}(s)}\right\|
< C. 
\]
\end{proof}
We now have all the tools to prove Theorem~\ref{thm:rank_reduction}. The proof is structured as follows
\begin{enumerate}
    \item[0.] Choosing $T>0$ such that $K_T$~\eqref{def:K_T} zeros do not intersect the zeros of $\det \Lap q$.
    \item Deriving the equations characterizing bad gains $(u, v)\in \R^m\times \R^{d-1}$.
    \item Choosing $u\in \R^m$ outside a finite union of hyperplanes.
    \item Choosing $v\in \R^{d-1}$ outside a finite set that depends on $u$.
\end{enumerate}
\begin{proof}
Let $\varepsilon>0$. We divide the proof into several steps.

\medskip\noindent\textbf{Step 0: choice of $T$.}
We first choose $T$ so that the zeros of $K_T$ in $\mathbb{C}_{\omega_2}$ do not
intersect $Z_q$.
We have
\small
\[
Z_T := \{s\in\mathbb{C}_{\omega_2} : K_T(s)=0\}
= \left\{s\in\mathbb{C}_{\omega_2}\setminus\{0\} : Ts\in 2\pi i\mathbb{Z}\right\}.
\]
\normalsize
By Proposition~\ref{prop:Z0_fini}, $Z_q$ is finite.
The bad choices of $T$ are those for which $Z_T\cap Z_q\neq\varnothing$, i.e.,
\begin{align*}
\mathcal{T}_{\mathrm{bad}}
:&= \left\{T>0 : \exists\,s\in Z_q,\ Ts\in 2\pi i\mathbb{Z}^\star\right\}
\\&= \bigcup_{s\in Z_q\cap i\mathbb{R}^\star}\frac{2\pi}{|\Im s|}\,\mathbb{N}^\star.
\end{align*}
Since $Z_q$ is finite, $\mathcal{T}_{\mathrm{bad}}$ is countable and not dense
in $\mathbb{R}_+$. We therefore choose $T\in(0,\varepsilon)\setminus\mathcal{T}_{\mathrm{bad}}$,
so that $Z_T\cap Z_q=\varnothing$.

\medskip\noindent\textbf{Step 1: characterization of bad gains.}
A pair $(u,v)\in\mathbb{R}^m\times\mathbb{R}^{d-1}$ is \emph{bad} if there exist
$s\in\mathbb{C}_{\omega_2}$ and $\eta\in\mathbb{C}^m\setminus\{0\}$ such that
$\eta^\star M_{u,v,T}(s)=0$, i.e.,
\begin{equation}\label{eq:bad_gains_eta}
\begin{cases}
\eta^\star\widehat{q}(s) - K_T(s)\,\eta^\star\widehat{p}_d(s)\,u^\star = 0,\\[0.15cm]
\eta^\star\widehat{p}_j(s) - v_j\,\eta^\star\widehat{p}_d(s) = 0,
\qquad j=1,\ldots,d-1.
\end{cases}
\end{equation}
We claim that $\eta^\star\widehat{p}_d(s)\neq 0$. Indeed, if
$\eta^\star\widehat{p}_d(s)=0$, then~\eqref{eq:bad_gains_eta} gives
$\eta^\star[\widehat{q}(s),\widehat{p}_1(s),\ldots,\widehat{p}_d(s)]=0$,
contradicting~\eqref{eq:ide-rank-ass}. Setting $w := \eta/(\eta^\star\widehat{p}_d(s))$, so that $w^\star\widehat{p}_d(s)=1$,
system~\eqref{eq:bad_gains_eta} becomes
\begin{equation}\label{eq:characterization_bad_s_w}
\begin{cases}
w^\star\widehat{q}(s) = K_T(s)\,u^\star,\\[0.10cm]
w^\star\widehat{p}_j(s) = v_j, \qquad j=1,\ldots,d-1,\\[0.10cm]
w^\star\widehat{p}_d(s) = 1.
\end{cases}
\end{equation}

\medskip\noindent\textbf{Step 2: exclusion of bad $u$ associated with $Z_q$.}
Suppose $s\in Z_q$. Since $Z_T\cap Z_q=\varnothing$, we have $K_T(s)\neq 0$,
so the first equation in~\eqref{eq:characterization_bad_s_w} gives
$u^\star = K_T(s)^{-1}w^\star\widehat{q}(s)$, i.e., $u^\star$ lies in the row space
of $\widehat{q}(s)$.
Since $s\in Z_q$, the matrix $\widehat{q}(s)$ is singular; let
$a\in\ker\widehat{q}(s)\setminus\{0\}$.
For any real $u$ in the row space of $\widehat{q}(s)$, we have $u^\star a=0$,
hence
\[
u^\star\Re(a)=0, \qquad u^\star\Im(a)=0.
\]
Since $a\neq 0$, at least one of $\Re(a)$, $\Im(a)$ is nonzero, so $u$ belongs
to a hyperplane of $\mathbb{R}^m$.
Since $Z_q$ is finite, all bad $u$ associated with points of $Z_q$ are contained
in a finite union of hyperplanes $\mathcal{F}_{\mathrm{bad}}\subsetneq\mathbb{R}^m$.
We choose $u\in\mathbb{R}^m\setminus\mathcal{F}_{\mathrm{bad}}$; with this choice,
no loss of rank occurs at any $s\in Z_q$.

\medskip\noindent\textbf{Step 3: exclusion of the remaining bad $v$.}
For $s\in\mathbb{C}_{\omega_2}\setminus Z_q$, the matrix $\widehat{q}(s)$ is
invertible, and the first equation in~\eqref{eq:characterization_bad_s_w} gives
\begin{equation}\label{eq:w_expression}
w^\star = K_T(s)\,u^\star\widehat{q}^{-1}(s).
\end{equation}
Combined with the normalization $w^\star\widehat{p}_d(s)=1$, this requires
\begin{equation}\label{eq:psi_u_def}
\psi_u(s) := K_T(s)\,u^\star\widehat{q}^{-1}(s)\widehat{p}_d(s) - 1 = 0.
\end{equation}
The function $\psi_u$ is holomorphic on $\mathbb{C}_{\omega_2}\setminus Z_q$
and is not identically zero: by Proposition~\ref{prop:Q_inv_bounded},
$\widehat{q}^{-1}$ and $\widehat{p}_d$ are bounded on
$\mathbb{C}_{\omega_2}\setminus K_{\omega_2}$, while $|K_T(s)|\to 0$ as
$\Re(s)\to+\infty$ or $|\Im(s)|\to+\infty$, so $|\psi_u(s)|\to 1$ in these
limits.
In particular, $\psi_u\not\equiv 0$, its zeros are isolated, and they are all
contained in a compact subset of $\mathbb{C}_{\omega_2}\setminus Z_q$; hence
\[
Z_{\psi_u} := \{s\in\mathbb{C}_{\omega_2}\setminus Z_q : \psi_u(s)=0\}
\]
is finite.
For each $s\in Z_{\psi_u}$, equations~\eqref{eq:w_expression}
and~\eqref{eq:characterization_bad_s_w} force
\[
v_j = K_T(s)\,u^\star\widehat{q}^{-1}(s)\widehat{p}_j(s),
\qquad j=1,\ldots,d-1.
\]
Hence the set of bad gains $v$ is contained in the finite set
\begin{align*}
&\mathcal{L}_{\mathrm{bad}}(u)
\\
&:= \left\{
\begin{pmatrix}
K_T(s)\,u^\star\widehat{q}^{-1}(s)\widehat{p}_1(s)\\
\vdots\\
K_T(s)\,u^\star\widehat{q}^{-1}(s)\widehat{p}_{d-1}(s)
\end{pmatrix}
\in\mathbb{R}^{d-1}
:
s\in Z_{\psi_u}
\right\}.
\end{align*}
We choose $v\in\mathbb{R}^{d-1}\setminus\mathcal{L}_{\mathrm{bad}}(u)$; with
this choice, no loss of rank occurs at any $s\in\mathbb{C}_{\omega_2}\setminus Z_q$.

\medskip\noindent\textbf{Conclusion.}
With
\[
T\in(0,\varepsilon)\setminus\mathcal{T}_{\mathrm{bad}},\quad
u\in\mathbb{R}^m\setminus\mathcal{F}_{\mathrm{bad}},\quad
v\in\mathbb{R}^{d-1}\setminus\mathcal{L}_{\mathrm{bad}}(u),
\]
no $s\in\mathbb{C}_{\omega_2}$ satisfies $\rank M_{u,v,T}(s)<m$,
so $\rank M_{u,v,T}(s)=m$ for all $s\in\mathbb{C}_{\omega_2}$.
\end{proof}
 \begin{rem} \label{rem:valid_rank_reduction}
At each iteration of Algorithm~\ref{algo:gains_rank_reduction}, Propositions~\ref{prop:Z0_fini} and~\ref{prop:Q_inv_bounded}
remain valid, so the rank condition is preserved throughout the reduction.
\end{rem}
\section{Network related results}
\subsection{Proof of Lemma~\ref{lem:cut_odd}}
\label{appendix:proof_lem_cut}
\begin{proof}
{We aim to show how any port-network can be transformed, via a
cutting operation, into one whose associated signed graph
is balanced.
To do so, we cut every subsystem of the port-network using the transform $\mathcal P$ defined for all $w_{i_0} \in L^2((0,1), \R^{n_i+m_i})$, for almost all $x\in (0,1)$, 
\begin{align*}
\mathcal{P} w_{i_0}(x)
:&=
\left(
\frac{1}{\sqrt{2}}\,w_{i_0}\!\left(\frac{x}{2}\right),\;
\frac{1}{\sqrt{2}}\,w_{i_0}\!\left(\frac{x+1}{2}\right)
\right)
\\&=: (w_{i_1}(x), w_{i_2}(x)), \quad \textit{}
\end{align*}
where $w_{i_k} = (w_{i_k}^+, w_{i_k}^-)^\star$ for $k=1,2$.
The inverse is defined as follows
\[
\mathcal{P}^{-1}(w_{i_1},w_{i_2})(x)
=
\begin{cases}
\sqrt{2}\,w_{i_1}(2x), & x\in\bigl(0,\tfrac{1}{2}\bigr),\\[1mm]
\sqrt{2}\,w_{i_2}(2x-1), & x\in\bigl(\tfrac{1}{2},1\bigr).
\end{cases}
\]
Hence $\mathcal{P}$ is bounded, linear, and invertible.
Moreover,
\begin{equation*}
\|\mathcal{P} w_{i_0}\|^2
= \|w_{i_0}\|^2,
\end{equation*}
hence $\mathcal{P}$ is an isometric isomorphism. Under the rescaling $\tilde{x}=2x$ on $(0,\frac{1}{2})$ and
$\tilde{x}=2x-1$ on $(\frac{1}{2},1)$, the two new subsystems $i_1$ and $i_2$
are again of the form~\eqref{eq:system_i_graph} with dimensions $n_{i_0}$
and $m_{i_0}$.
The coupling between $i_1$ and $i_2$ at the midpoint is given by the identity:
\begin{equation}
w_{i_1}^+(1) = w_{i_2}^+(0),
\qquad
w_{i_2}^-(0) = w_{i_1}^-(1).
\end{equation}
Since port $p_{i_1}^{-1}$ (at $x=1$) is connected to port $p_{i_2}^{+1}$
(at $x=0$), the new edge $(i_1,i_2)$ has signature
$\sigma(i_1,i_2)=(-1)(+1)=-1$. 
We now apply this cutting operation to every subsystem $i\in \{1, \cdots, n_g\}$.
The resulting associated signed graph
$\widetilde \graph$ is balanced.
Let $C=i_0|\cdots|i_{q-1}|i_0$ be a cycle of $\graph$. At each vertex $i_r$, denote by $\ell_r$ and $k_r$ the
ports label used by the incoming and outgoing edges of $i_r$,
respectively. By Definition~\ref{def:signature}, the cycle signature is
\[
\sigma(C)=\prod_{r=0}^{q-1}\ell_r k_r=(-1)^{r_C},
\]
where $r_C$ is the number of vertices at which
$\ell_r\neq k_r$. After cutting, the new cycle $\widetilde C$ traverses a
new internal edge, of signature $-1$, exactly at those $r_C$
vertices. Consequently,
\[\sigma(\widetilde C)
=\sigma(C)(-1)^{r_C}
=(-1)^{2r_C}=+1.
\]
Hence every cycle of the resulting graph $\widetilde \graph$ has positive
signature, and the graph is balanced.}
 \end{proof}

By Lemma~\ref{lem:cut_odd}, we may assume without loss of generality that
the signed graph $\mathcal{G}$ is balanced.
\subsection{Proof of Theorem~\ref{thm:graph_main}}
\label{section:proof_thm_graph_main}
\begin{proof}
Under Assumption~\ref{assump:pairwise_consistency}, the signed graph associated
with the port-network~\eqref{eq:system_i_graph}-\eqref{eq:boundary_couplings_graph}
is well defined.
Moreover, by Lemma~\ref{lem:cut_odd}, up to applying finitely many unitary
cutting operators, we may assume the associated signed graph is balanced.
By Lemma~\ref{lem:good_folding}, there exists a good folding choice
$m\in\{-1,+1\}^{n_g}$.
For each subsystem $i$, define the folding operator $\mathcal{F}_m$, for all $w\in L^2((0,1), \R^{n_i+m_i})$
\[
(\mathcal{F}_m w)_i
= \begin{pmatrix}\widetilde{w}_i^+\\ \widetilde{w}_i^-\end{pmatrix},
\]
where
\begin{align*}
m(i)=+1 &\implies
\widetilde{w}_i^+(x):=w_i^+(x),\widetilde{w}_i^-(x):=w_i^-(x),\\
m(i)=-1 &\implies
\widetilde{w}_i^+(x):=w_i^-(1-x),\\&\widetilde{w}_i^-(x):=w_i^+(1-x).
\end{align*}
Thus
\[
\mathcal{F}_m w
= \col(\widetilde{w}_1^+,\dots,\widetilde{w}_{n_g}^+,
       \widetilde{w}_1^-,\dots,\widetilde{w}_{n_g}^-).
\]
The map $\mathcal{F}_m$ is linear; since it consists only of component permutations and spatial reflections $x\mapsto 1-x$, it is bounded on the natural $L^2$-state space. Its inverse is obtained by applying the same operation to the folded components, so $\mathcal{F}_m$ is a bounded linear isomorphism. After folding, the components satisfy
\begin{align*}
m(i)=+1 &\implies \widetilde{w}_i^+\in\mathbb{R}^{n_i},\quad \widetilde{w}_i^-\in\mathbb{R}^{m_i},\\
m(i)=-1 &\implies \widetilde{w}_i^+\in\mathbb{R}^{m_i},\quad \widetilde{w}_i^-\in\mathbb{R}^{n_i}.
\end{align*}
Therefore the total dimensions of $\widetilde{w}^+$ and $\widetilde{w}^-$ are
$
n := \sum_{i:\,m(i)=+1} n_i + \sum_{i:\,m(i)=-1} m_i,~
m := \sum_{i:\,m(i)=+1} m_i + \sum_{i:\,m(i)=-1} n_i,
$
and $n+m = \sum_{i=1}^{n_g}(n_i+m_i)$.
Set
\begin{align*}
\widetilde{w}^+ :&= \col(\widetilde{w}_1^+,\dots,\widetilde{w}_{n_g}^+)\in\mathbb{R}^n,\\
\widetilde{w}^- :&= \col(\widetilde{w}_1^-,\dots,\widetilde{w}_{n_g}^-)\in\mathbb{R}^m,~
\widetilde{w} := \col(\widetilde{w}^+,\widetilde{w}^-).
\end{align*}

The folded in-domain input terms are defined as follows:
\begin{align*}
m(i)=+1 &\implies
\widetilde{h}_i^+(x):=h_i^+(x),~\widetilde{h}_i^-(x):=h_i^-(x),\\
m(i)=-1 &\implies
\widetilde{h}_i^+(x):=h_i^-(1-x),~ \widetilde{h}_i^-(x):=h_i^+(1-x),
\end{align*}
so that
\begin{align*}
\widetilde{h}^+(x) :&= \col(\widetilde{h}_1^+(x),\dots,\widetilde{h}_{n_g}^+(x))
\in\mathbb{R}^{n\times d},\\
\widetilde{h}^-(x) :&= \col(\widetilde{h}_1^-(x),\dots,\widetilde{h}_{n_g}^-(x))
\in\mathbb{R}^{m\times d}.
\end{align*}
The folded velocity (spatially varying here) and coupling matrices are obtained analogously:
\begin{align*}
m(i)=+1 &\implies
\widetilde{\Lambda}_i^\pm(x):=\Lambda_i^\pm(x),\\
&\widetilde{\Sigma}_i^{\ell k}(x):=\Sigma_i^{\ell k}(x),~
\ell,k\in\{+,-\},\\
m(i)=-1 &\implies
\widetilde{\Lambda}_i^+(x):=\Lambda_i^-(1-x), ~
\widetilde{\Lambda}_i^-(x):=\Lambda_i^+(1-x),
\end{align*}
and
\begin{align*}
&\widetilde{\Sigma}_i^{++}(x):=\Sigma_i^{--}(1-x),~
\widetilde{\Sigma}_i^{+-}(x):=\Sigma_i^{-+}(1-x),\\
&\widetilde{\Sigma}_i^{-+}(x):=\Sigma_i^{+-}(1-x),\quad
\widetilde{\Sigma}_i^{--}(x):=\Sigma_i^{++}(1-x).
\end{align*}

Before folding, the boundary equation is (see~\eqref{eq:boundary_couplings_graph})
\[
r = Ky + \Bgraph\,U(t).
\]
The trace of $\mathcal{F}_m$ induces block permutation matrices $P_r(m)$ and
$P_y(m)$ such that $\widetilde{r}=P_r(m)r$ and $\widetilde{y}=P_y(m)y$, giving
\[
\widetilde{r} = \widetilde{K}\,\widetilde{y} + \widetilde{\mathcal{B}}\,U(t),
~
\widetilde{K} := P_r(m)\,K\,P_y(m)^{-1},~
\widetilde{\Bgraph} := P_r(m)\,\Bgraph.
\]
Since $m$ is a good folding choice, all nonzero couplings connect ports with
the same final label, so $\widetilde{K}_{+-}=\widetilde{K}_{-+}=0$.
The folded boundary variables can therefore be ordered as
\[
\widetilde{r} = \begin{pmatrix}\widetilde{w}^+(t,0)\\ \widetilde{w}^-(t,1)\end{pmatrix},
\quad
\widetilde{y} = \begin{pmatrix}\widetilde{w}^-(t,0)\\ \widetilde{w}^+(t,1)\end{pmatrix},
\quad
\widetilde{K} = \begin{pmatrix}Q & 0\\ 0 & R\end{pmatrix}.
\]
Writing \[
\widetilde{\Sigma}^{\ell k}(x)
=
\diag\bigl(
\widetilde{\Sigma}_1^{\ell k}(x),
\dots,
\widetilde{\Sigma}_{n_g}^{\ell k}(x)
\bigr),
\qquad \ell,k\in\{+,-\},
\] and $\widetilde{\mathcal{B}}=\col(B_0,B_1)$, we have
$Q\in\mathbb{R}^{n\times m}$, $R\in\mathbb{R}^{m\times n}$,
$B_0\in\mathbb{R}^{n\times d}$, $B_1\in\mathbb{R}^{m\times d}$.
Consequently, the folded system takes the form~\eqref{eq:hyperbolic_couple}.
\end{proof}
\section{Notations}
\label{appendix:notations}
\begin{table}[H]

              \begin{minipage}{\columnwidth}

                            \begin{center}

                                          \begin{tabular}{ll}

                    $w$ & state of the original system~\eqref{eq:hyperbolic_couple}.\\
                    $U$ & control input.\\
                    $\mathcal T$ & backstepping transform from Lemma~\ref{lem:backsteppin_transform}.\\
                    $\gamma=(\alpha, \beta) =\mathcal T w$ & state of the target system~\eqref{eq:hyperbolic_target_2}\\
                    $\spaceY $ & $L^2((0,1), \R^{dm+1})$\\
                    $\gains$ & controller gains\\

                    $\gainspace$ & controller gains space defined in~\eqref{eq:gainspace}\\

                    $\dmax$ & maximal transport delay defined in~\eqref{eq:retard_max}.\\

                    $X_t$ & state of the IDE~\eqref{eq:ide-input-target}.\\

                    $p_i^{\pm 1}$ & boundary port of a network~\eqref{eq:system_i_graph}.\\

                    $\graph=(\mathcal{V},\mathcal{E},\sigma)$ &  signed graph, see Definition~\ref{def:sign_graph}.\\

                    $(m(1),\dots,m(n_g))$ & folding choice, see Definition~\ref{def:folding_choice}.\\

                    $i_0|\cdots |i_q$ & path on a signed graph.\\

                    $\mathcal{X}_g$ &$ \prod_{i=1}^{n_g}(L^2([0,1],\mathbb{R}))^{n_i+m_i}$.\\

                                                         \bottomrule

                                          \end{tabular}

                            \end{center}

                            \bigskip\centering

              \end{minipage}

    \caption{Notations}

    \label{tab: notations}

\end{table}



\bibliographystyle{IEEEtran}
\bibliography{references}
\end{document}